\documentclass[a4paper]{amsart}

\usepackage{a4wide}
\usepackage{amsmath,amssymb}
\usepackage{enumerate}

\newcommand{\cB}{\mathcal{B}}
\newcommand{\cF}{\mathcal{F}}
\newcommand{\cG}{\mathcal{G}}
\newcommand{\cH}{\mathcal{H}}
\newcommand{\cK}{\mathcal{K}}
\newcommand{\cM}{\mathcal{M}}
\newcommand{\cO}{\mathcal{O}}
\newcommand{\cP}{\mathcal{P}}
\newcommand{\cS}{\mathcal{S}}
\newcommand{\cT}{\mathcal{T}}
\newcommand{\cW}{\mathcal{W}}
\newcommand{\cX}{\mathcal{X}}
\newcommand{\eps}{\varepsilon}
\newcommand{\R}{\mathbb{R}}
\newcommand{\res}{\mathop{\hbox{\vrule height 7pt width .5pt depth 0pt \vrule height .5pt width 6pt depth 0pt}}\nolimits}
\newcommand{\supp}{\operatorname{supp}}

\theoremstyle{plain}
	\newtheorem{theorem}{Theorem}[section]
	\newtheorem{proposition}[theorem]{Proposition}
	\newtheorem{corollary}[theorem]{Corollary}
	\newtheorem{lemma}[theorem]{Lemma}
\theoremstyle{definition}
	\newtheorem{definition}[theorem]{Definition}
\theoremstyle{remark}
	\newtheorem{remark}[theorem]{Remark}

\title[A model for Alzheimer's disease]{Well-posedness of a mathematical model for Alzheimer's disease}

\author[M. Bertsch]{Michiel Bertsch}
\address{Dipartimento di Matematica, Universit\`{a} di Roma ``Tor Vergata'', Via della Ricerca Scientifica 1, 00133 Roma, Italy \newline
		\indent Istituto per le Applicazoni del Calcolo ``M. Picone'', Consiglio Nazionale delle Ricerche, Via dei Taurini 19, 00185 Roma, Italy}
\email{bertsch@mat.uniroma2.it}
\author[B. Franchi]{Bruno Franchi}
\address{University of Bologna, Department of Mathematics, Piazza di Porta S. Donato 5, 40126 Bologna, Italy}
\email{bruno.franchi@unibo.it}
\author[M. C. Tesi]{Maria Carla Tesi}
\address{University of Bologna, Department of Mathematics, Piazza di Porta S. Donato 5, 40126 Bologna, Italy}
\email{mariacarla.tesi@unibo.it}
\author[A. Tosin]{Andrea Tosin}
\address{Department of Mathematical Sciences ``G. L. Lagrange'', Politecnico di Torino, Corso Duca degli Abruzzi 24, 10129 Torino, Italy}
\email{andrea.tosin@polito.it}

\keywords{Transport and diffusion equations; Smoluchowski equations; mathematical models of Alzheimer's disease}

\begin{document}

\begin{abstract}
We consider the existence and uniqueness of solutions of an initial-boundary value problem for a coupled system of PDE's arising in a model for Alzheimer's disease. Apart from reaction diffusion equations, the system contains a transport equation in a bounded interval for a probability measure which is related to the malfunctioning of neurons. The main ingredients to prove existence are: the method of characteristics for the transport equation, a priori estimates for solutions of the reaction diffusion equations, a variant of the classical contraction theorem, and the Wasserstein metric for the part concerning the probability measure. We stress that all hypotheses on the data are not suggested by mathematical artefacts, but are naturally imposed by modelling considerations. In particular the use of a probability measure is natural from a modelling point of view. The nontrivial part of the analysis is the suitable combination of the various mathematical tools, which is not quite routine and requires various technical adjustments.
\end{abstract}

\maketitle

\section{Introduction}\label{intro}
In \cite{BFMTT, BFTT} a macroscopic mathematical model was proposed which describes the onset and evolution of Alzheimer's disease (AD). This model is meant to mirror the so-called \emph{Amyloid Cascade Hypothesis~\cite{haass_delkoe,karran_et_al,SH}}, coupled with the spreading of the disease through neuron-to-neuron transmission (prionoid hypothesis~\cite{Braak_DelTredici,tatarnikova_et_al}). Alzheimer's disease (AD) is the prevalent form of late life dementia. Its global prevalence, about 24 millions in 2011, is expected to double in 20 years~\cite{reitz_etal_epidemiology}.

In order to clarify the structure of our equations and the choice of our assumptions, let us sketch a gist of their biological background. We refer to \cite{BFMTT,BFTT} for a complete description of the model and an account of the most recent biomedical literature. The model focusses on the role of the polymer beta-amyloid, in particular its toxic soluble isoform A$\beta_{42}$. Monomers of  A$\beta_{42}$ are regularly produced by neurons and successively cleared -- among others -- by the microglia. In the last decades, researchers have observed that an imbalance between production and clearance of A$\beta_{42}$ (shortly A$\beta$ from now on) is a very early, often initiating factor in AD. Soluble A$\beta$  diffuses through the microscopic tortuosity of the brain tissue and undergoes an agglomeration process. Eventually this leads to the formation of long, insoluble fibrils, accumulating in spherical deposits known as senile plaques that are solid deposits of even larger aggregates of A$\beta$; nowadays, plaques are not considered neurotoxic, but are usually abundantly present in the brain of AD-patients (though they can be present in old brains without any symptom of dementia). Plaques are clinically observable through medical imaging using a special type of PET  (Positron Emission Tomography) scan.

Below we briefly describe the model. The main purpose of the present paper is to establish its mathematical well-posedness. Mathematically, our model consists of a transport equation coupled with a system of nonlinear diffusion equations (a Smoluchowski-type system with diffusion). Due to the very nature of the biological  phenomena we are studying, the main feature of such a system is that the transport velocity depends on the solution of the Smoluchowski equation, which, in turn, contains a source term that depends on the solution of the transport equation, so that the two groups of equations cannot be uncoupled. For an introduction to the use of transport equations in mathematical models of life sciences, we refer the reader to \cite{perthame}.

Let us give a less cursory description of the system which we consider, from both a mathematical and biomedical point of view. We do not enter the biological details and merely mention  those which are related to the structure of our equations. We would like  to stress that the equations of our model involve only functions that have a precise qualitative clinical counterpart in routinely observable phenomena: the health state of the different brain regions (by means of a PET measuring the cerebral glucose metabolism), the amount of A$\beta$ in the cerebral spinal fluid, and the A$\beta$ plaques (by means of amyloid-PET scans).

\medskip

Let $\Omega\subseteq\R^n$ be a portion of cerebral tissue. The molar concentration of soluble A$\beta$ polymers of length $m$ at $x\in \Omega$ and time $t\ge 0$ is denoted by $u_m(x,t)$ ($1\leq m<N$), that of clusters of oligomers of length greater or equal to $N$ (the plaques) by $u_N(x,t)$. We use a parameter $a$, ranging from $0$ to $1$, to describe the \emph{degree of malfunctioning} of a neuron; $a$ close to $0$ stands for ``the neuron is healthy'' and $a$ close to $1$ for ``the neuron is dead''. Given $x\in\Omega$ and $t\geq 0$, $f=f_{x,t}$ is a probability measure and $df_{x,t}(a)$ denotes the fraction of neurons at $x$ and time $t$ with degree of malfunctioning between $a$ and $a+da$. The progression of AD is mainly determined by the \emph{deterioration rate} of the health state of the neurons, $v=v_{x}(a,t)\ge 0$. We use the notation $v[f]$ to stress its dependence on $f$:
\begin{equation}
	(v[f])_{x}(a,t):=\int_{\Omega}\left(\int_{[0,1]}\cK(x,a,y,b)
	\,df_{y,t}(b)\right)dy+\cS(x,a,u_1(x,t),\dots,u_{N-1}(x,t)).
	\label{velocity}
\end{equation}

The term $\cS\geq 0$ in~\eqref{velocity} models the action of toxic A$\beta$ oligomers. For example, assuming that the toxicity of soluble A$\beta$-polymers is proportional to their total mass and introducing a threshold value $\overline{U}>0$ for the amount of toxic A$\beta$ needed to damage neurons, a possible choice for $\cS$ is
\begin{equation} \label{mathcal S}
	\cS=C_{\mathcal S}(1-a){\left(\sum\limits_{m=1}^{N-1}mu_m(x,\,t)-\overline{U}\right)}^{+},
		\qquad\text{where } p^+:=\max\{p,0\},
\end{equation}
see also~\cite{BFMTT} for a more detailed discussion.

The integral term in \eqref{velocity} describes the possible propagation of AD through the neural pathway. Malfunctioning neighbours are harmful for a neuron's health state, while healthy ones are not: $\cK(x,a,y,b) \geq 0$ for all $x,y\in \Omega$ and $a,b\in [0,1]$ and 
$$ \cK(x,a,y,b) = 0 \quad \text{if\ } a>b. $$
For the sake of simplicity we choose $\cK(x,a,y,b)=\cG_x(a,b)h(\vert x-y\vert)$, where $h(r)$ is a nonnegative and decreasing function which vanishes at some $r=r_0$ and satisfies $\int_{\vert y\vert<r_0}h(\vert y\vert)\,dy=1$. For instance, in \cite{BFMTT} the following form of $\cG_x$ is used: $\cG_x(a,\,b)=C_{\mathcal G}(b-a)^{+}$, which does not depend explicitly on $x$. In the limit $r_0\to 0$,~\eqref{velocity} reduces to
\begin{equation} \label{velocity bis}
	(v[f])_{x}(a,t)=\int_{[0,1]}\cG_x(a,b)\,df_{x,t}(b)+\cS(x,a,u_1(x,t),\,\dots,u_{N-1}(x,t)).
\end{equation}
We shall henceforth use~\eqref{velocity bis} for the deterioration rate $v[f]$.

In view of the meaning of the rate $v$, the equation for $f$ is given by
\begin{equation} \label{transport}
	\partial_t f+\partial_a(fv[f])=J[f].
\end{equation}
The term $J[f]$ represents the onset of AD: we assume that in small (randomly chosen) parts of the cerebral tissue, concentrated for instance in the hippocampus and described by a characteristic function $\chi(x,t)$, the degree of malfunctioning of neurons randomly jumps to higher values due to external agents or genetic factors. More precisely, $(J[f])_{x,t}$ denotes the measure
\begin{equation} \label{def J}
	d(J[f])_{x,t}(a):=\eta(t)\chi(x,t)\left[\left(\int_{[0,1]}P(t,b, a)\,df_{x,t}(b)\right)da-	df_{x,t}(a)\right],
\end{equation}
where the function $P(t,b, a)$ is the probability to jump from state $b$ to state $a$ (which vanishes if $a<b$) and $\eta>0$ is the jump frequency. A possible choice is 
\begin{equation*}
	P(t,b,a)\equiv P(b,a)=
		\begin{cases}
			\tfrac{2}{1-b} & \text{if\ } b\leq a\leq\frac12(1+b) \\
			0 & \text{otherwise}.
		\end{cases}
\end{equation*}

It is worth stressing that the choice of looking for a \emph{measure} $f_{x,t}$ comes from the model itself. In fact, a ``healthy brain'' would correspond to $f_{x,t}(a)=\delta(a)$, where $\delta$ is the Dirac measure centred at the origin.

Now we are ready to write the system of equations for $f$, $u_1,\cdots,u_N$:
\begin{equation} \label{complete system}
\begin{cases}
	\partial_t f+\partial_a\left(fv[f]\right)=J[f] & \text{in } \Omega\times [0,1]\times (0,T] \\
	\eps\partial_tu_1-d_1\Delta u_1=R_1:=-u_1\sum\limits_{j=1}^{N}a_{1,j}u_j+\cF[f]-\sigma_1u_1 & \text{in } Q_T=\Omega\times (0,T] \\
	\eps\partial_tu_m-d_m\Delta u_m=R_m:=\tfrac{1}{2}\sum\limits_{j=1}^{m-1}a_{j,m-j}u_ju_{m-j} \\
		\phantom{\eps\partial_tu_m=d_m\Delta u_m=R_m:=}-u_m\sum\limits_{j=1}^{N}a_{m,j}u_j-\sigma_m u_m & \text{in } Q_T \quad (2\leq m<N) \\
	\eps\partial_tu_N=\tfrac{1}{2}\sum\limits_{\substack{j+k\geq N \\ k,\,j<N}}a_{j,k}u_ju_{k} & \text{in }Q_T.
\end{cases}
\end{equation}
Here $\eps>0$ is a small parameter which expresses the existence of two time scales: processes which determine the dynamics of A$\beta$ (production, aggregation, diffusion, deposition) occur on a much smaller time scale (hours) than the evolution of the disease (years). The diffusion coefficients $d_m$ depend on the length of the polymer (longer polymers diffuse less); plaques do not diffuse. The quadratic terms (in $u_i$) model the aggregation of A$\beta$ polymers, according to the Smoluchowski equations. We refer to \cite{AFMT,FT_torino} for an extensive discussion of the aggregation mechanism and the choice of the coagulation rates $a_{i,j}$. The linear terms $-\sigma_mu_m$ model the phagocytic activity of the microglia and other bulk clearance processes~\cite{Iliff_et_al}.

We stress that system~\eqref{complete system} is fully coupled, because the transport equation for $f$ contains a dependence on $u_1,\,\dots,\,u_{N-1}$ in the deterioration rate $v[f]$, cf. \eqref{velocity bis}. Notice that if $\cS\equiv 0$ in \eqref{velocity bis} then the equation for $f$ decouples from the rest of the system and may be possibly studied alone by relying on the results reported in~\cite{CCGU,EHM}. Nevertheless, the assumption $\cS\equiv 0$ is not a minor issue in the modelling of AD spreading, because it would imply a propagation of the disease due only to prionic diffusion, which is a controversial topic in the medical literature. For this reason, in our model we prefer to take into account also the toxic contribution of A$\beta$ oligomers, i.e. $\cS\neq 0$ in \eqref{velocity bis}, which requires to study system \eqref{complete system} as a whole.

A$\beta$ monomers are produced by neurons. Their production increases if neurons are damaged, and a possible choice for  the source term $\cF$ in the equation for $u_1$ is
\begin{equation} \label{mathcal F}
	\cF[f](x,t)=C_{\mathcal F}\int_0^1(\mu_0+a)(1-a)\,df_{x,t}(a).
\end{equation}
The small constant $\mu_0>0$ accounts for A$\beta$ production by healthy neurons. 

We assume that $\partial \Omega$ consists of smooth disjunct boundaries, $\partial \Omega_0$ and $\partial \Omega_1$, where $\partial \Omega_1$ represents the disjunct union of the boundaries of the cerebral ventricles through which A$\beta$ is removed from the cerebrospinal fluid (CSF) by an outward flow through the choroid plexus (cf. \cite{Iliff_et_al,serot_et_al}). In the present paper we solve system \eqref{complete system} with appropriate initial-boundary conditions:
\begin{equation} \label{IBC}
	\begin{cases}
		f_{x,0}=(f_0)_x & \text{if } x\in\Omega \\
		u_i(x,0)=u_{0i}(x) &\text{if } x\in\Omega,\ 1\le i\le N\\
		\partial_n u_i(x,t)=0 & \text{if } x\in\partial\Omega_0, \ t>0, \ 1\le i<N \\
		\partial_n u_i(x,t)=-\gamma_i u_i(x,t) & \text{if } x\in\partial\Omega_1,\ t>0,\ 1\le i<N,
	\end{cases}
\end{equation}
where $n$ is the outward pointing normal on $\partial \Omega$.

In Section \ref{Section2} we describe the hypotheses on the data and formulate the main result on global well-posedness. In Section \ref{characteristics} we rewrite the system in terms of the characteristics of the transport equation for $f$ and show that the new system is equivalent to the original one. We point out that under our assumptions the characteristics exist in the classical sense. The major difficulty arises from the strong nonlinearity of the system: the transport equation depends nonlinearly on both its solution (through an integral operator) and the solution of the Smoluchowski system, which in turn depends on the solution of the transport equation. In Section \ref{section local existence} we use a contraction argument to prove local existence and uniqueness; not surprisingly, the metric for the probability measures $f$ will involve Wasserstein distances. The fact that the Wasserstein distance $\cW_1$ depends on the action of the measures on \emph{Lipschitz} functions yields a technical difficulty when we try to apply an iteration argument in order to obtain the local existence of a solution. This difficulty can be bypassed thanks to an {\it ad hoc} formulation of the standard fixed point theorem. Finally, in Section \ref{section global existence} we prove a priori bounds which imply global existence. In Appendix \ref{app:measures} we collect some technical facts about probability measures and Wasserstein distances to make this paper as self-consistent as possible.

We conclude with some comments.

For more details on the model we refer to \cite{BFMTT}, which also contains some two-dimensional simulations and qualitative comparison with experimental data. A derivation of the transport equation for $f$ from a Boltzmann-type kinetic approach is contained in \cite{BFTT}.

The macroscopic model which we have briefly described, and in particular the use of the Smoluchowski equations to model the aggregation of A$\beta$, is based on an earlier microscopical model described in \cite {AFMT,FT_torino}. The passage from that microscopic aggregation-diffusion model to Smoluchowski equations with a source term as in \eqref{complete system} is obtained by a two-scale homogenization technique in \cite{FL,FL_wheeden}.

For the moment the model deliberately neglects some important aspects of the disease such as the role of the tau-protein, but in a future paper we shall extend the model and make it more realistic. Nevertheless, the term $\cG$ in the deterioration rate for the equation of $f$ can already be thought of as taking into account a ``prionic'' model of the spread  of the disease, and associated with the diffusion of the intraneural tau-protein (see, e,g., \cite{HHPW,tatarnikova_et_al}). Despite the extreme complexity of AD, mathematically such an extension is doable due to the high degree of flexibility of the modelling approach. The major difficulty is the lack of both experimental data and detailed knowledge about the relevant biomedical processes, but fortunately biomedical research on AD evolves rapidly. 

In this context we also mention a recent paper by Hao and Friedman \cite{HF}, which does take into account a higher degree of AD's complexity and contains simulations of medical therapies; the authors however do not consider AD's initial stage. A major challenge is how to diagnose AD's early stage and develop therapies to slow down its further development.      

\section{Problem statement and main results} \label{Section2}
Throughout the paper we set $T>0$, $N\in\mathbb N$, while $\Omega \subset \R^n$ is an open and bounded set with a smooth boundary $\partial \Omega$, which is the disjunct union of $\partial \Omega_0$ and $\partial \Omega_1$. 

To treat the measures $f_{x,t}$ we introduce a metric space $X_{[0,1]}$: 
\begin{definition} \label{wasserstein metric} 
The space $\mathcal P([0,1])$ of probability measures on $[0,1]$ endowed with the Wasserstein distance $\cW_1$ is denoted by $X_{[0,1]}$.
\end{definition}

We refer to \cite{AGS} for the definition of the Wasserstein distances $\cW_p$. By Proposition \ref{moment}, $X_{[0,1]}$ is a complete separable metric space. By Proposition \ref{wasserstein}, a sequence $(\mu_n)_{n\in\mathbb N}$ converges in $X_{[0,1]}$ if and only if it converges narrowly or $\text{weakly}^*$. 

We denote by $C([0,T]; X_{[0,1]})$ the space of continuous functions from the interval $[0,T]$ to $X_{[0,1]}$. Endowed with the distance
$$ \max_{0\le t\le T} \cW_1((\mu_1)_t, (\mu_2)_t), $$
also $C([0,T]; X_{[0,1]})$ is a complete metric space.

\subsection{Hypotheses on the data}\label{subsection2.1}
Throughout the paper we shall use the following assumptions on the data (below $\partial_a$, $\nabla_u$ etc. denote distributional derivatives; $C$ denotes a generic constant):
\begin{itemize}
\item [($H_1$)] $\varepsilon, C_{\mathcal F}, \mu_0, d_i, \sigma_i, \gamma_i, a_{i,j}$ are positive constants ($1\le i<N$, $1\le j\le N$);
\item [($H_2$)] $u_{0i}\in C(\overline \Omega)$ is nonnegative ($i=1,\cdots,N$), and $(f_0)_x \in X_{[0,1]}$ for a.e.~$x\in \Omega$;
\item [($H_3$)] $\chi$ is the characteristic function of a measurable set $Q_0\subseteq Q_T=\Omega\times [0,T]$; the function $\eta \in C([0,T])$ is nonnegative;
\item [($H_4$)] for a.e.~$x\in \Omega $, $\cG_x\in C\left([0,1]^2\right)$, $\mathcal G_x(1,b)= 0$ for $b\in [0,1]$, and 
\begin{equation} \label{Ga global}
	-C\le\partial_a\cG_x\le 0,\,\left|\partial_b\cG_x\right|\le C \quad \text{in } [0,1]^2;
\end{equation}
\item [($H_5$)] $\cS\in L^\infty\left(\Omega; C\left([0,1] \times [0,\infty)^{N-1}\right)\right)$, $\mathcal S(x,1,u_1,\dots, u_{N-1})=0$ for $u_i\ge 0$ and a.e. $x\in \Omega $, and for all compact sets $K\subset [0,\infty)^{N-1}$ there exists a constant $C(K)$ such that for a.e. $x\in\Omega$ 
\begin{equation}\label{Sa global}
-C(K)\le \partial_a \mathcal S(x,a,u) \le 0
, \ \left| \mathcal \nabla_u\cS(x,a,u)  \right|\le C(K)\quad \text{for } a\in[0,1], \, u\in K;
\end{equation}
\item [($H_6$)] $P\in C([0,T]\times [0,1]^2)$, $P$ is nonnegative, for all $t\in [0,T]$
\begin{equation}\label{hyp on P}
\int_0^1P(t,b, a)\,da=1 \quad\text{for }b\in [0,1]\,, \qquad P(t,b,a)=0 \quad \text{if }a<b
\end{equation}
and there exists $L>0$ such that for all $a',a'',b',b''\in [0,1]$ and $t\in [0,T]$
\begin{equation}\label{P lip}
\big| P(t,b',a')-P(t,b'',a'')\big| \le L\, \big( \big|b' - b''\big|+ \big|a' - a''\big|\big).
\end{equation}
\end{itemize}

\subsection{Main result}\label{subsection main result}
We introduce some additional notation. Let $\cM(0,1)$ be the space of signed Radon measures on the interval $(0,1)$. Then  $\cM(0,1)$ is the dual of $C([0,1])$, and  
$\mu :\Omega\times (0,T)\to \cM(0,1)$ is said to be weakly$^*$ measurable if for any $\rho\in C([0,1])$ the map 
\begin{equation} \label{weak*}
	(x,t)\mapsto \int \rho(a) d\mu_{x,t}(a)
\end{equation}
is measurable in  $\Omega\times (0,T)$. We say that  
$$ f \in \mathcal L(\Omega;  C([0,T]; X_{[0,1]})) $$ 
if $f \in C([0,T]; X_{[0,1]}))$  for a.e.~$x\in \Omega$ and $f$ is weakly$^*$ measurable as a function from $\Omega\times (0,T)$ in $\cM(0,1)$. In particular, if $f\in \mathcal L(\Omega;  C([0,T]; X_{[0,1]}))$, then, by the Fubini-Tonelli Theorem, for all $\psi\in C([0,1]\times \overline \Omega\times [0,T])$
$$ x\mapsto \int_0^T\left(\int \psi(a,x,t)\,df_{x,t}(a)\right)dt \quad\text{belongs to }L^\infty(\Omega). $$

\begin{definition}\label{definition of solution 2}
A $(N+1)$-ple $(f,u_1,\cdots,u_N)$ is called a solution of problem \eqref{complete system}-\eqref{IBC} in $[0,T]$ if
\begin{itemize}
\item[$(i)$] $f\in\mathcal{L}(\Omega; C([0,T]; X_{[0,1]}))$;
\item[$(ii)$] $u_i\in C(\overline Q_T)$ and $u_i\ge 0$ in $Q_T$ for $1\le i\le N$;
\item[$(iii)$] the first equation in \eqref{complete system} is satisfied in a  weak sense: for a.e. $x\in\Omega$ 
$$ \int_0^\tau\left(\int (\partial_t\phi+v_x\partial_a\phi)\,df_{x,t}+\int\phi\,dJ_{x,t}\right)dt
	=\int\phi(\cdot,\tau)\,df_{x,\tau}-\int\phi(\cdot,0)\,d(f_0)_x $$
for all $\tau\in [0,T]$ and $\phi\in C^1([0,1]\times [0,T])$, where the function $v$ is defined by \eqref{velocity bis} and the signed measure $J$ by \eqref{def J};
\item[$(iv)$] if $1\le i<N$, $u_i\in L^2([0,T]; H^1(\Omega))$ and
\begin{equation}
	\begin{split}
		d_i\int_0^T & \left[\int_\Omega\nabla u_i(x,s)\cdot\nabla\psi(x,s)dx+\gamma_i\int_{\partial\Omega_1}u_i(x,s)\psi(x,s)d\sigma(x)\right]ds \\ &
		\qquad =\eps\iint_{Q_T}u_i\psi_t+\eps\int_\Omega u_{0i}\psi(x,0)\,dx+\iint_{Q_T}R_i\psi 
	\end{split}
\end{equation}
for all $\psi\in H^1([0,\tau]; H^1(\Omega))$, $\psi(x,\tau)=0$, where $R_i$ is defined as in  \eqref{complete system} and $\mathcal F$ (which is part of $R_1$) by \eqref{mathcal F};
\item[$(v)$] $\partial_t u_{N}\in C(\overline Q_T)$, $u_N(\cdot,0)=u_{0N}$ in $\Omega$, and the equation for $u_N$ in \eqref{complete system} is satisfied in $Q_T$.
\end{itemize}
\end{definition}

\begin{remark}\label{remark def}
\noindent {\bf (a)} It follows from \eqref{def J} and $(H_6)$ that, for a.e. $x\in \Omega$, $\displaystyle{\int} dJ_{x,t}=0$ for $t\in [0,T]$.
\noindent {\bf (b)} The concept of weak solution of the first order transport equation, defined in Definition \ref{definition of solution 2}$(iii)$, needs some explanation. It follows from \eqref{Ga global}-\eqref{Sa global} that,  for a.e.~$x\in \Omega$, $v$ is uniformly Lipschitz continuous with respect to $a$:
\begin{equation}\label{v_a}
	\partial_av_{x}(a,t)=\!\!\int_{[0,1]} \!\partial_a\cG(x,a,b)\, df_{x,t}(b)+\partial_a\cS(x,a,u_1(x,t),\dots,u_{N-1}(x,t))\le 0.
\end{equation}
In particular it follows from $(H_{4-5})$ that, for a.e.~$x\in \Omega$, $v_{x}(1,t)=0$ for $t\in [0,T]$ and $v_{x}(a,t)\ge 0$ for $a\in [0,1]$ and $t\in [0,T]$. This implies that formally the ``flux'' $fv$ vanishes at $a=1$, a condition which is made precise by the choice of continuous test functions $\phi(x,a,t)$ without any restriction at $a=1$. Since $v\ge 0$ at $a=0$, characteristics (see the next section) ``enter the domain $[0,1]$'' at $a=0$; so we need a boundary condition at $a=0$ which, according to Definition \ref{definition of solution 2}$(iii)$, is again the no-flux condition. Actually this is imposed by the condition that $f_{x,\tau}$ is a probability measure in $[0,1]$: choosing $\phi\equiv 1$ it follows from Definition \ref{definition of solution 2}$(iii)$ and Remark \ref{remark def}(a) that for a.e.~$x\in \Omega$
$$ \int df_{x,\tau}=\int  d(f_0)_x =1 \quad \text{for }\tau\in (0,T]. $$
\end{remark}

The main result of the paper is the following well-posedness result.
\begin{theorem}\label{theorem main result}
Let  $\Omega \subset \R^n$ be an open and bounded set with a smooth boundary $\partial \Omega$, which is the disjunct union of $\partial \Omega_0$ and $\partial \Omega_1$. Let $T>0$ and $N\in\mathbb N$, and let hypotheses $(H_{1-6})$ be satisfied. Then  problem \eqref{complete system}-\eqref{IBC} has a unique solution in $[0,T]$ in the sense of Definition \ref{definition of solution 2}.
\end{theorem}

\section{The characteristics}\label{characteristics}
Let $f\in\mathcal{L}(\Omega;C([0,T]; X_{[0,1]}))$ and $u_i\in C(\overline Q_T)$, and let $v[f]$ be defined by \eqref{velocity bis}. By the Lipschitz continuity of $a\mapsto v_{x}(a,t)$ (see Remark \ref{remark def}(b)), for a.e. $x\in\Omega$ the problem for the characteristic issued from $y\in [0,1]$,
\begin{equation}\label{char}
	\begin{cases} 
		\partial_t A_x(y,t)=v_{x}(A_x(y,t),t) & \text{for } 0<t\le T \\
		A_x(y,0)=y
	\end{cases}
\end{equation}
has a unique solution which satisfies
\begin{equation}\label{properties A}
	\begin{cases}
		0\le A_x(y_1,t)<A_x(y_2,t)\le A_x(1,t)=1 & \text{if } 0\le y_1<y_2\le 1,\,0\le t\le T \\
		A_x(y,t_1)\le A_x(y,t_2) &\text{if } y\in [0,1],\, 0\le t_1\le t_2\le T\,.
	\end{cases}
\end{equation}
Observe that, for a.e.~$x\in \Omega$, the function $y\mapsto A_x(y,t)$ is continuous and
\begin{equation}\label{y-derivative A}
\partial_y A_x(y,t) = \exp\left(\int_0^t \partial_av_{x}(A_x(y,s),s)\,ds\right)>0 \quad \text{for all } t\in [0,T].
\end{equation}
In particular for a.e.~$x\in \Omega$ the function $y\mapsto A_x(y,t)$ is injective for all $t\in [0,T]$.

Below we shall reformulate problem \eqref{complete system}-\eqref{IBC} in terms of the characteristics, but before doing so we prove the following result.

\begin{proposition} \label{supports} 
Let $f\in\mathcal{L}(\Omega;C([0,T];X_{[0,1]}))$ and $u_i\in C(\overline Q_T)$. Let $v[f]$ and $J[f]$ be defined by \eqref{velocity bis} and \eqref{def J}. Let, for a.e.~$x\in \Omega$, $A_x(y,t)$ be the solution of \eqref{char} for any $y\in[0,1]$. If $f$ satisfies \eqref{transport} in the sense of Definition \ref{definition of solution 2}\,$(iii)$, then, for a.e.~$x\in \Omega$,
\begin{equation}\label{claim bis}
	\supp f_{x,t},\ \supp J_{x,t}\subseteq [A_x(0,t),1]\quad\text{for }t\in (0,T].
\end{equation}
\end{proposition}

\begin{proof} 
$A_x$ is well defined for a.e.~$x\in \Omega$. We fix such $x$ and also $\tau\in (0,T]$. Let $h\in C^1(\R)$ be nondecreasing and satisfy $h\equiv 0$ in $(-\infty,0]$ and $h\equiv 1$ in $[1,\infty)$. Let $\delta>0$ and set for a.e.~$x\in \Omega$
$$
h_\delta(s)=h(s/\delta) \ \text{for }s\in \R; \quad
\psi_\delta(a,t)=h_\delta(A_x(0,t)-a)\ \text{for } a\in[0,1], \, t\in[0,T].
$$
Then $\psi_\delta$ is of class $C^1$ and
$$
\partial_a\psi_\delta=-\frac 1\delta h'\left(\frac {A_x(0,t)-a}{\delta}\right), \quad \partial_t\psi_\delta
=\frac 1\delta 
v_{x}(A_x(0,t),t)
h'\left(\frac {A_x(0,t)-a}{\delta}\right).
$$

We use $\psi_\delta$ as a test function in Definition \ref{definition of solution 2}\,$(iii)$. Since $A_x(0,0)=0$, $\psi_\delta(a,0)=0$ if $a\ge 0$ and
$\int \psi_\delta(\cdot,0)\,d(f_0)_x=0$. Therefore the test function relation implies that
\begin{equation}\label{july 11 eq:3} 
\int \psi_\delta(\cdot,\tau)
\,df_{x,\tau}
- \int_0^\tau\left(\int \psi_\delta 
\, dJ_{x,t}
\right)dt\to 0 \quad\text{as }\delta\to 0
\end{equation}
if we prove that  
\begin{equation}\label{july 11 eq:2}
\int_0^\tau\left(\int (\partial_t\psi_\delta+v_x\partial_a\psi_\delta)
\, df_{x,t}
\right)dt
\to 0\qquad \mbox{as $\delta \to 0$.}
\end{equation}

To prove \eqref{july 11 eq:2} we observe that 
$$
|\partial_t\psi_\delta+v_x\partial_a\psi_\delta|=\left|\frac {
v_{x}(A_x(0,t),t)\!-\!v_{x}(a,t)
}\delta h'\left(\frac {A_x(0,t)\!-\!a}{\delta}\right)\right|
\le C\sup_{s\in \R}|sh'(s)|
$$
for some constant $C$ which does not depend on $\delta$, whence
\begin{equation}\label{july 11, eq:1}
\begin{split}
&\left|
\int_0^\tau \left(\int (\partial_t\psi_\delta+v_x\partial_a\psi_\delta)
\, df_{x,t}
\right)dt\right|
\\&\qquad \le
C \int_0^\tau\left(\int 
df_{x,t}
\res(A_x(0,t)-\delta,A_x(0,t))\cap[0,A_x(0,t))\right)dt.
\end{split}
\end{equation}
Here and in the following, the symbol $\res$ denotes the restriction of a measure to a measurable subset, see \cite[Definition 1.8]{mattila}. Since $\bigcap\limits_{\delta>0} \left(A_x(0,t)-\delta,A_x(0,t)\right)\cap[0,A_x(0,t))= \emptyset$ and
$$
\left| \int 
df_{x,t}
\res (A_x(0,t)-\delta,A_x(0,t))\cap[0,A_x(0,t))\right| \le 1 \quad \text{for }t\in [0,\tau],
$$
\eqref{july 11 eq:2} follows from \eqref{july 11, eq:1} and the Dominated Convergence Theorem.

By \eqref{july 11 eq:3} and the Dominated Convergence Theorem,
\begin{equation}\label{july 17 eq:1}
\int 
df_{x,\tau}
\res [0,A_x(0,\tau))= \int_0^\tau\left(\int 
dJ_{x,t}
\res[0,A_x(0,t))\right)dt\,.
\end{equation}
It follows from \eqref{def J},  the Fubini-Tonelli Theorem and \eqref{hyp on P} that
\begin{equation*}
\begin{aligned}
&\int  
dJ_{x,t}(a)
\res [0,A_x(0,t)) \\
&\quad =\eta(t)\chi(x,t)
\!\left[\int \!\!\left(\int_{0}^{A_x(0,t)} \!\!\!P(t,b, a)\,da\!\right) \!
d f_{x,t}(b)
\!-\!\!\int \!
df_{x,t}(a)
\res [0,A_x(0,t))\right]\\
& \quad =  \eta
\chi\left[ \int\left(\int_b^{A_x(0,t)}\!\!\!P(t,b, a)\, da \right)
d f_{x,t}(b)
\res[0,A_x(0,t))   -\int 
df_{x,t}(a)
\res [0,A_x(0,t))\right]\\
& \quad \le  \! \eta\chi\!\left[ \int \! 
d f_{x,t}(b)
\res [0,A_x(0,t))\!-\!\int \!
df_{x,t}(a)
\res [0,A_x(0,t))\right]\!=0.\\
\end{aligned}
\end{equation*}
Combined with \eqref{july 17 eq:1}, this implies  \eqref{claim bis}.
\end{proof}

We now reformulate the original problem in terms of the characteristics. Specifically, we shall see below that the measure $f$ can be obtained by transporting along the characteristics a suitable measure $g$ (i.e., $f$ is the push forward of $g$ through $A$, cf. Definition \ref{def:push_forward}), which satisfies:
\begin{equation}\label{system rewritten}
\begin{cases}
\medskip \partial_t 
A_x(y,t)
\!=\!\!\displaystyle {\int}\!
\cG_{x}(A_{x}(y,t), A_{x}(\xi,t))
\,dg_{x,t}(\xi)
\!+\!\cS(x, A_x(y,t),u_1,\dots,u_{N-1})\\
\smallskip
\partial_t 
g_{x,t}(y)
=\eta\chi \! \left[\partial_y 
A_{x}(y,t)
\displaystyle{\int}  \!\! 
P(t,A_{x}(\xi,t),  A_{x}(y,t))
\,dg_{x,t}(\xi)\!-\!g_{x,t}(y)
\right]\\
\smallskip \eps\partial_tu_1\!-\!d_1\Delta u_1\!=\tilde R_1:=-u_1\!\sum\limits_{j=1}^{N}a_{1,j}u_j\!-\!\sigma_1u_1
+C_{\mathcal F}\!\displaystyle {\int}\!(\mu_0\!+\! A_{x}(\xi,t))(1\!-\! A_{x}(\xi,t)) 
\,dg_{x,t}(\xi)\\
\smallskip \eps\partial_tu_m-d_m\Delta u_m=\tilde R_m:=\tfrac{1}{2}\sum\limits_{j=1}^{m-1}a_{j,m-j}u_ju_{m-j} 
-u_m\sum\limits_{j=1}^{N}a_{m,j}u_j-\sigma_m u_m  \\
\eps\partial_tu_N=\tfrac{1}{2}\sum\limits_{\substack{j+k\geq N \\ k,\,j<N}}a_{j,k}u_ju_{k},\\
\end{cases}
\end{equation}
where $x\in\Omega$, $y\in[0,1]$, $t\in (0,T]$ and $2\leq m<N$, with initial-boundary conditions
\begin{equation}\label{IC+BC}
\begin{cases}
\smallskip 
g_{x,0}(y)=(f_0)_x(y),
\  
A_{x}(y,0)
= y &\text{if }x\in \Omega, \ 0 \le y\le 1\\
\smallskip u_i(x,0)=u_{0i}(x) &\text{if }x\in \Omega, \ 1\le i\le N\\
\smallskip \partial_n u_i(x,t)=0&\text{if }x\in \partial \Omega_0, \ t\in (0,T], \ 1\le i< N\\
\partial_n u_i(x,t)=-\gamma_i u_i(x,t)&\text{if }x\in \partial \Omega_1,\ t\in (0,T], \ 1\le i< N.
\end{cases}
\end{equation}

\begin{definition}\label{definition of solution}
The $(N+2)$-ple $(A,g,u_1,\cdots,u_N)$ is called a solution of problem \eqref{system rewritten}-\eqref{IC+BC} in $[0,T]$ if
\begin{itemize}
\item[$(i)$] $g\in\mathcal{L}(\Omega;C([0,T];X_{[0,1]}))$;
\smallskip
\item[$(ii)$] $A,\partial_tA\in L^\infty(\Omega; C([0,1]\times [0,T];[0,1]))$;
\smallskip
\item[$(iii)$] $u_i\in C(\overline Q_T)$ and $u_i\ge 0$ in $Q_T$ for $1\le i\le N$;
\smallskip
\item[$(iv)$] for a.e.~$x\in \Omega$,  $A_x$ satisfies \eqref{system rewritten}$_1$
and $A_{x}(y,0)=y$ for $y\in [0,1]$;
\item[$(v)$] 
equation \eqref{system rewritten}$_2$ for $g$ is satisfied in a weak sense for a.e.~$x\in \Omega$: for all $\tau\in (0,T]$
and $\phi
 \in C([0,1]\times [0,T])$ with $\partial_t \phi\in C([0,1]\times [0,T])$ 
\begin{equation}\label{weak sol measures}\begin{split}
&\quad\int\!
\phi(y,\tau)
\,dg_{x,\tau}(y)
\! -\! \int\!
\phi(y,0)
\,d(f_0)_x(y)
\!-\! \int_0^\tau\!\!\left(  \int \partial_t \phi(y,t)
\, dg_{x,t} (y)
\!\right)\! dt
\\&
\quad=\!\int_0^\tau \!\! \!\eta
\chi
\!\left[\int_0^1 \phi(y,t)\partial_y 
A_{x}(y,t)
\! \left(\displaystyle{\int}  \!\! P(t,A_{x}(\xi,t),  A_{x}(y,t))
 \,dg_{x,t}(\xi)
 \!\right)\!dy
\!-\!\!\!\int\! 
\phi(y,t)
\, dg_{x,t} (y)
\right]\!\!dt;
\end{split}\end{equation}
\item[$(vi)$] if $1\le i<N$, $u_i\in L^2([0,T]; H^1(\Omega))$ and
\begin{equation}\begin{split} 
d_i \int_0^T&
\left[ \int_\Omega \!\nabla  u_i(x,s) \!\cdot\! \nabla \psi(x,s) dx| 
+  \gamma_i\!\int_{\partial\Omega_1}\!\!\!  u_i(x,s)\psi (x,s) d\sigma(x) \right] ds
\\&
\qquad =\eps\!\iint_{Q_T}  \!\!u_i\partial_t \psi +\eps\int_\Omega u_{0i}\psi(x,0)\, dx + \iint_{Q_T}  \tilde R_i \psi 
\end{split}\end{equation}
for all $\psi\in H^1([0,\tau]; H^1(\Omega))$, $\psi(x,\tau)=0$, where $\tilde R_i$ is defined as in  \eqref{system rewritten};
\smallskip
\item[$(vii)$] $\partial_t u_{N}\in C(\overline Q_T)$, $u_N(\cdot,0)=u_{0N}$ in $\Omega$, and the equation for $u_N$ 
in  \eqref{system rewritten} is satisfied in $Q_T$. 
\end{itemize}
\end{definition}

In the remainder of this section we prove the equivalence of problems \eqref{complete system}-\eqref{IBC} and \eqref{system rewritten}-\eqref{IC+BC}. The following result is a first step in this direction.

\begin{theorem}\label{th equiv 2} 
Let hypotheses $(H_{1-6})$ be satisfied. Let $(A,g,u_1,\cdots,u_N)$ be a solution of 
\eqref{system rewritten}-\eqref{IC+BC} in $[0,T]$ and  set, for a.e.~$x\in \Omega$,
$$ 
f_{x,t}
:=A_x(\cdot,t)_\# 
g_{x,t}
\quad\text{for all }t\in [0,T]. $$
Then $(f,u_1,\dots,u_N)$ is a solution of problem \eqref{complete system}-\eqref{IBC} in $[0,T]$. 
\end{theorem}
\begin{proof}
Since, for a.e.~$x\in \Omega$, 
$g_{x,t}$
is a Borel regular probability measure in $[0,1]$ for $t\in [0,T]$, so is 
$f_{x,t}$.
By \eqref{y-derivative A}, for a.e.~$x\in \Omega$ the function $y\to  A_x(y,t)$ is injective for $t\in [0,T]$, so that, by \cite{mattila}, Theorem 1.18, 
\begin{equation}\label{july 19 eq:1}
\mathrm{supp}\, 
f_{x,t}
= A_x(\mathrm{supp}\,
g_{x,t}
,t)\subseteq A_x([0,1],t)=[A_x(0,t),1].
\end{equation}
In particular $
f_{x,t}\in X_{[0,1]}$ for a.e.~$x\in \Omega$. In addition, by Remark \ref{remark def}\,(b) and Proposition \ref{cont meas}, the map $t \mapsto 
f_{x,t}
$ belongs to $C([0,T]; X_{[0,1]})$ for a.e.~$x\in\Omega$.

Let, for a.e.~$x\in \Omega$, $v$ be defined by \eqref{velocity bis} and $J$  by \eqref{def J}. 
By \eqref{july 19 eq:1} and \eqref{hyp on P}, 
$\displaystyle{\int}  P(t,b, a)
\, df_{x,t}(b)
= 0$ if $a<A_x(0,t)$, whence
\begin{equation}\label{supp J}
\supp 
J_{x,t}
\subset [A_x(0,t),1].
\end{equation}

To avoid cumbersome notations, we set $B_x(\cdot, t):= A_x^{-1}(\cdot, t)$. Since 
$$
A_x(\cdot,t): [0,1] \to [A_x(0,t), A_x(1,t)]=[A_x(0,t), 1],
$$
$B_x(\cdot, t)$ is well defined in $[A_x(0,t),1]$, $ B_x(A_x(y,t),t)\equiv y$ for $y\in[0,1]$,
and $ A_x(B_x(a,t),t) \equiv a$ for $a\in [A_x(0,t), 1]$.
Since        
$\mathrm{supp}\,
f_{x,t}
\subset [A_x(0,t),1]$, integrals of functions of $B(\cdot,t)$ with respect to
$f_{x,t}$
are well defined.  

By Definition \ref{definition of solution}\,$(iv)$, $\partial_t  A_x(y,t)=
v_{x}(A_x(y,t),t)
$ for a.e. $x\in \Omega$.
By \eqref{y-derivative A}, $B_x$ is Lipschitz continuous with respect to $y$ for a.e. $x\in \Omega$.
Differentiating the identity $A_x(B_x(y,t),t)=y$ with respect to $t$ and $y$, we obtain that 
\begin{equation}\label{5_2 eq:5}
\begin{cases}
\partial_y A_x(B_x(a,t),t)\partial_t B_x(a,t) =-\partial_t  A_x(B_x(a,t),t)
%\\ \qquad \qquad\qquad \qquad \qquad \qquad 
= -v_{x}(A_x(B_x(a,t)),t)  = - v_{x}(a,t)
\\
\partial_y   A_x(B_x(y,t),t)\partial_y B_x(y,t)=1,
\end{cases}\end{equation}
so that
$\partial_t B_x(y,t)\partial_y   A_x(B_x(y,t),t)\partial_y B_x(y,t) 
= \partial_t B_x(y,t)$ and, by  \eqref{5_2 eq:5},
\begin{equation}\label{5_2 eq:8}
\partial_t B_x(y,t)=-v_{x}(y,t) \partial_y B_x(y,t).
\end{equation}

Let $\psi\in C^1([0,1] \times [0,T])$. Let $x$ be fixed such that $A_x,\partial_t A_x\in C([0,T]; X_{[0,t]})$, and set 
$$ \phi(y,t)=\psi(A_x(y,t),t)\quad \text{for }y\in [0,1] $$
and 
\begin{equation*}
C_{\phi}=-\!\int\!\phi(y,\tau) 
\,dg_{x,\tau}(y)
+\! \int\!\phi(y,0) 
\,d(f_0)_x(y)
= -\!\int\!\phi(B_x(a,\tau),\tau) 
\,df_{x,\tau}(a)
+ \!\int\!\phi(a,0) 
\, d(f_0)_x(a).
\end{equation*}
Since $\phi$ satisfies the conditions in Definition \ref{definition of solution}$(v)$, it follows that
\begin{equation}\label{weak sol}\begin{split}
&- \int_0^T\left(  \int \phi_t(y,t) \, dg_{x,t} (y)
\right) dt
\\&
\ =\!\int_0^T \!\! \!\eta\chi\!\left[\int_0^1\phi(y,t)\partial_y A_x(y,t)\!
 \left(\displaystyle{\int}  \!\! P(t,A_x(\xi,t),A_x(y,t))
 \, dg_{x,t}(\xi)
 \!\right)\!dy\!
- \!\!  \int\! \phi(y,t)  \,dg_{x,t} (y)
\right]dt\!+\! C_\phi
\\&
\ =\!\int_0^T \!\! \!\eta\chi\!\!\left[\int_0^1\!\! \phi(B_x(A_x(y,t),t)\partial_y A_x(y,t)\!\!
 \left(\displaystyle{\int}  \!\! P(t,b,A_x(y,t))
 \, df_{x,t}(b)
 \!\!\right)\!dy\right.
 \\&
\qquad\qquad\qquad\qquad\qquad\qquad\qquad\qquad\qquad\qquad\qquad - \!\left.  \int\! \phi(B_x(A_x(y,t),t)
\,dg_{x,t} (y)
\right]dt+ C_\phi
\\&
\ =\!\int_0^T \!\! \!\eta\chi\!\!\left[\int_{ A_x(0,t)}^{ A_x(1,t)}\!\! \phi(B_x(a,t),t) \! \left(\displaystyle{\int}  \!\! P(t,b, a)
\,df_{x,t}(b)
\right)da\!
- \!\!  \int\! \phi(B_x(A_x(y,t),t)
\,dg_{x,t} (y)
\right]dt\!+\! C_{\phi},
\end{split}\end{equation}
where we have used Theorem 1.19 in \cite{mattila} and the relation $da =\partial_y A(x,y,t)\, dy$. 

On the other hand, the left-hand side of \eqref{weak sol} can be written as
\begin{equation}\label{5_2 eq:3}
- \int_0^T\left(  \int \phi_t(y,t) \,dg_{x,t} (y)\right) dt
= -  \int_0^T\left(  \int \phi_t(B_x(a,t),t) \,df_{x,t} (a)
\right) dt.
\end{equation}

Let $a\in [ A_x(0,t),1]$. Then
$\psi(a,t)= \phi(B_x(a,t),t)$
and 
\begin{equation}\label{5_2 eq:1}
\partial_t \psi (a,t) =
\partial_y\phi (B_x(a,t),t)\partial_t B_x(a,t)
+ \phi_t(B_x(a,t),t).
\end{equation}
Since, by \eqref{july 19 eq:1}, $\supp 
f_{x,t}
\subseteq [A_x(0,t),1]$, it follows from  \eqref{weak sol}-\eqref{5_2 eq:1} that 
\begin{equation}\label{5_2 eq:4}\begin{split}
- \int_0^T &\left(\int \partial_t \psi (a,t)\,df_{x,t}(a)
\right) dt =
-\int_0^T \left(\int \partial_y\phi (B_x(a,t),t)\partial_t B_x(a,t)\,df_{x,t}(a)
\right) dt 
\\&+
\int_0^T \eta\chi\left[
\int_{ A_x(0,t)}^{ A_x(1,t)}\phi(B_x(a,t),t) \! \left(\displaystyle{\int}  \!\! P(t,b, a)
\,df_{x,t}(b)
\right)da\right.
\\&
\qquad\qquad \qquad\qquad\qquad\qquad\qquad \qquad \qquad -   \left.\int \phi(B_x(a,t),t)   \,df_{x,t} (a)
\right]\, dt + C_\phi
 \\=&
-\int_0^T \left(\int \partial_y\phi (B(x,a,t),t)\partial_t B(x,a,t)
\,df_{x,t}(a)
\right) dt 
\\&+
\int_0^T \!\eta \chi\!
\left[\int_{A_x(0,t)}^{A_x(1,t)}\! \psi(a,t) \! 
\left(\displaystyle{\int}  \!\! P(t,b, a)\,df_{x,t}(b)\right)
\!da-\!\int \!\psi(a,t)  \,df_{x,t}(a)\right]
\, dt + C_\phi
\\=&
-\int_0^T \!\!\left(\int \!\partial_y\phi (B_x(a,t),t)\partial_t B_x(a,t)\,df_{x,t}(a)\!\right) dt \!
+\!\! \int_0^T \!\left( \int \!\psi(a,t) \,dJ_{x,t}(a)\! \right) dt + C_\phi.
\end{split}\end{equation}
By \eqref{5_2 eq:8},
$$
\begin{aligned}
\int \partial_y\phi (B_x(a,t),t)\partial_t B_x(a,t)\,df_{x,t}(a)
&=\int \partial_y\phi (B_x(a,t),t) \partial_a B_x(a,t)  
v_{x}(a,t) \,df_{x,t}(a)
\\& 
=\int \partial_a \psi (a,t)   v_{x}(a,t) \,df_{x,t}(a),
\end{aligned}
$$
whence, by \eqref{5_2 eq:4},  the first equation in \eqref{complete system} is satisfied in the sense of distributions:
\begin{equation*}\begin{split}
&- \int_0^T \!\!\left(\int \partial_t \psi (a,t)\,df_{x,t}(a)
\!\right)\!\! dt 
=\!\int_0^T\! \!\left(\int \partial_a \psi (a,t)   v_{x}(a,t)\,df_{x,t}(a)
\right) dt 
\\&\qquad\qquad+\int_0^T \left( \int \psi(a,t) \,dJ_{x,t}(a)
\right)\, dt
\!-\!\int\!\psi(a,\tau) \,df_{x,\tau}(y)
\!+\! \int\!\psi (a,0) \,d(f_0)_x(a).
\end{split}\end{equation*}

Concerning the Smoluchowski system in \eqref{complete system},
ii is enough to observe that the third equation in \eqref{system rewritten} and the second
equation in \eqref{complete system} coincide, since
\begin{equation*}\begin{split}
\displaystyle {\int}(\mu_0+ A_x(\xi,t))(1- A_x(\xi,t))\,dg_{x,t}(\xi)
=\displaystyle {\int}(\mu_0+a)(1-a)
\,df_{x,t}(a).
\end{split}\end{equation*}
Here we have used again Theorem 1.19 in \cite{mattila}.
\end{proof}

The proof that problems \eqref{complete system}-\eqref{IBC} and \eqref{system rewritten}-\eqref{IC+BC}
are equivalent is completed by the following result.

\begin{theorem}\label{th equiv 1} Let $(f,u_1,\cdots,u_N)$ be a solution of 
\eqref{complete system}-\eqref{IBC} in $[0,T]$ and let $A_x(y,t)$ be defined
by \eqref{char}. Then there exists a probability measure $g_{x,t}$ such that
$$
f_{x,t}:= A_x(\cdot,t)_\# g_{x,t},
$$
and $(A,g,u_1,\dots,u_N)$ is a solution of problem \eqref{system rewritten}-\eqref{IC+BC} in $[0,T]$. 
\end{theorem}

\begin{proof} As before  we reason for a.e.~$x\in \Omega$. 
Fixing such  $x\in \Omega$, and also an arbitrary $t\in [0,T]$, we consider the map
$$
A_x(\cdot,t): [0,1]\mapsto [A_x(0,t),1].
$$
By Proposition \ref{supports}, $
f_{x,t} = f_{x,t}
\res [A_x(0,t),1]$. Hence, by Theorem 1.20 in \cite{mattila}, 
there exists a Radon measure 
$g_{x,t}$
on $[0,1]$ such that
$$
f_{x,t} = f_{x,t}
\res [A_x(0,t),1] =  A_x(\cdot,t)_\#  g_{x,t}.
$$
Obviously $g_{x,t}$ is a probability measure and belongs to $X_{[0,1]}$. By Corollary \ref{cont meas}, the map $t\mapsto g_{x,t}$ is continuous with respect to the Wasserstein metric.
In addition, 
$g_{x,t}\to (f_0)_x$ as $t\to 0$
since $A_x(y,0)=y$. Therefore $g$ satisfies the qualitative assumptions in order to be a solution of \eqref{system rewritten} and \eqref{IC+BC}.

To complete the proof of the theorem, it is enough to check the identities in the proof of Theorem \ref{th equiv 2} in the opposite direction.
\end{proof}

\section{Local existence and uniqueness}\label{section local existence}
By Theorems~\ref{th equiv 2} and \ref{th equiv 1}, problems  \eqref{complete system}-\eqref{IBC} and \eqref{system rewritten}-\eqref{IC+BC} are equivalent. In this section we prove  
local (w.r.t.~$t$) existence and uniqueness of a solution of problem \eqref{system rewritten}-\eqref{IC+BC}. In section \ref{section global existence} we shall show that this solution can be continued in $[0,T]$, which completes the proof of the main result, Theorem \ref{theorem main result}.

So in this section we have to prove:
\begin{theorem}\label{theorem local existence}
Let  $\Omega \subset \R^n$ be an open and bounded set with a smooth boundary $\partial \Omega$, 
which is the disjunct union of smooth manifolds $\partial \Omega_0$ and $\partial \Omega_1$. 
Let $T>0$ and $N\in\mathbb N$, and let hypotheses $(H_{1-6})$ be satisfied.
Then  there exists $\tau\in (0,T]$ such that  problem \eqref{system rewritten}-\eqref{IC+BC} has a unique solution in $[0,\tau]$.
\end{theorem}

The proof is based on a contraction argument. To this purpose we introduce a suitable metric space.

\begin{definition}\label{X spaces} Let $\tau\in (0,T]$ be given. 
We denote by $(\mathcal X_\tau,d)$ the complete metric space
\begin{equation*}
\mathcal X_\tau:=L^{\infty}(\Omega; C([0,1]\times [0,\tau]; [0,1]))\times 
\times
C(\overline{\Omega}\times [0,\tau];\mathbb R^{N}),
\end{equation*}
where $L^{\infty}(\Omega; C([0,1]\times [0,\tau]; [0,1]))$,
  and
$C(\overline{\Omega}\times [0,\tau];\mathbb R^{N})$ are endowed with their natural metrics
as normed spaces, and 
$\mathcal L(\Omega;  C([0,T]; X_{[0,1]}))$ is endowed
with the metric 
$$
\sup_{x\in\Omega}\max_{t\in [0,T]} \mathcal W_1(f_{x,t}, g_{x,t}) \,
$$
(notice that  condition \eqref{weak*} passes to the limit with respect to the $\mathcal W_1$-convergence, by Proposition \ref{wasserstein}).

We denote by $\mathcal X_{\tau,\rho}$ the closed ball in $\mathcal X_\tau$ of radius $\rho>0$ centered at $(y,f_0,u_0)$.
\end{definition}

Observe that, for the moment, we have given up the nonnegativity of $u_i$, which will be recovered during 
the proof of Theorem \ref{theorem local existence}.
For this reason we define $\mathcal S$ also for negative values of $u_i$, by requiring that $\mathcal S$ is even
with respect to $u_i$ for each $i=1,\cdots,N\!-\!1$.

We must construct the map to which we can apply the contraction argument. We shall do this step by step.

\begin{lemma} \label{char febr 7} Let $(\hat A,g,u) \in \cX_T$ and
set, for a.e.~$x\in \Omega$,
\begin{equation}\label{v first}%\begin{split}
\hat v_{x}(a,t) 
 := \displaystyle {\int} \cG_x(a,\hat A_x(\xi,t))\, dg_{x,t}(\xi)+\cS(x,a,u_1,\dots,u_{N-1})\ge 0.
 %\\&= \displaystyle {\int}_{[0,1]} \cG(x,a,\hat A(x,\xi,t))dg(x,\xi,t)+\cS(x,a,u_1,\dots,u_{N-1})  \ge 0.
%\end{split}
\end{equation}
Then, for a.e.~$x\in \Omega$, the Cauchy problem
\begin{equation}\label{cauchy A}
\begin{cases} 
\partial_t \underline A_x(y,t)=\hat v_{x}(\underline A_x(y,t),t) &\text{for }t>0\\
\underline A_x(y,0)=y\in [0,1]
\end{cases}
\end{equation}
has a unique solution defined for all $t\in (0,T]$, and
the function $y\mapsto \underline A_x(y,t)$ is continuous, strictly increasing (and thus open) on $[0,1]$ 
and maps $[0,1]$ onto $[\underline A_x(0,t),1]$ for all $t\in [0,T]$.
Finally, the map $(x,y,t) \mapsto \underline A_x(y,t)$ belongs to $L^{\infty}(\Omega; C([0,1]\times [0,T]; [0,1])$.
\end{lemma}

\begin{proof} We claim that, for a.e.~$x\in \Omega$, the map
$(a,t)\mapsto \hat v_x(a,t)$ is continuous and
Lipschitz continuous with respect to  $a\in [0,1]$, uniformly in $t\in [0,T]$.

By \eqref{Sa global} this is trivial for the map
$(a,t)\mapsto  \cS(x,a,u_1(x,t),\dots,u_{N-1}(x,t))$,
since $(x,t)\mapsto (u_1,\dots,u_{N-1})$ is continuous on $\overline \Omega\times [0,T]$
and $(u_1,\dots,u_{N-1})$ belongs to a compact set of $\mathbb R^{N-1}$. 
It remains to show that  
$(a,t)\mapsto   \displaystyle {\int}\cG_x(a,\hat A_x(\xi,t))\,dg_{x,t}(\xi)$
is continuous and uniformly Lipschitz continuous with respect to $a\in [0,1]$. 

Let $a_0,a \in [0,1]$ and $ t_0,t\in(0,T]$ be given. Then 
\begin{equation*}
\begin{split}
&\left| \displaystyle {\int}\cG_x(a,\hat A_x(\xi,t))\,dg_{x,t}(\xi,t)-\displaystyle {\int}\cG_x(a_0,\hat A_x(\xi,t_0))
\,dg_{x,t_0}(\xi)\right|
\\&\quad 
\le
\left| \displaystyle {\int}\cG_x(a,\hat A_x(\xi,t))\,dg_{x,t}(\xi)- 
\displaystyle {\int}\cG_x(a_0 ,\hat A_x(\xi,t_0))\,dg_{x,t}(\xi,t)\right|
\\&\qquad 
+\left| \displaystyle {\int}\cG_x(a_0 ,\hat A_x(\xi,t_0))\,dg_{x,t}(\xi)
- \displaystyle {\int}\cG_x(a_0,\hat A_x(\xi,t_0))\,dg_{x,t_0}(\xi)\right|
%\\&\quad
:= I_1+I_2.
\end{split}
\end{equation*}
Since 
$(a,\xi,t)\mapsto\cG_x(a,\hat A_x(\xi,t))$ is uniformly continuous in $[0,1]^2\times [0,T]$, 
$I_1\to 0$ as $(a,t)\to (a_0,t_0)$.
Since $\xi\mapsto\cG_x(a_0,\hat A_x(\xi,t_0))$
is continuous in $[0,1]$) and $t\mapsto g_{x,t}$ is narrowly continuous
(see Proposition \ref{wasserstein}), also $I_2\to 0$ as $t\to t_0$.

Similarly, by \eqref{Ga global}, for a.e.~$x\in\Omega$ and all $\xi\in [0,1]$ and $t\in [0,T]$,
$$ \left|\cG_x(a,\hat A_x(\xi,t)) - \cG_x(a',\hat A_x(\xi,t))\right|\le C|a-a'| \quad\text{for } a,a'\in [0,1]. $$

This completes the proof of the claim, which implies, for a.e.~$x\in \Omega$, the existence and uniqueness of the solution problem \eqref{cauchy A} for all $y\in [0,1]$. By a standard argument,
\begin{equation}\label{dAdy}
\partial_y \underline A_x(y,t)= \exp\left[ \int_0^t \partial_a\hat v_x(\underline A_x(y,s),s) ds\right]>0,
\end{equation}
so $0< C_1\le \partial_y \underline A_x(y,t)\le C_2$ for some constants $C_1$ and $C_2$ which depend 
on the compact set $K\subset \mathbb R^{N-1}$ which contains $(u_1(x,t),\dots,u_{N-1}(x,t))$.
\end{proof}

\begin{remark}\label{stime A} 
It follows from the proof of Lemma \ref{char febr 7} (in particular from \eqref{dAdy}) 
that $\underline A_x(\xi,s)$ is Lipschitz
continuous in $\xi$, uniformly with respect to $x$ and $s$.
\end{remark}

\begin{lemma} \label{g} Let $(\hat A,g,u) \in \cX_T$. Let, for a.e.~$x\in \Omega$, $\underline A$ be defined as in 
Lemma \ref{char febr 7}  and $(F[g])_{x,t}$ be the signed measure on $[0,1]$ defined by
\begin{equation*}
d(F[g])_{x,t}\!= \! \eta(t)\chi(x,t) \! \left[ \partial_y 
\underline{A}_x(y,t)\displaystyle{\int}  \!\! P(t, \underline A_x(\xi,t), \underline{A}_x(y,t)) dg_{x,t}(\xi)\,dy-dg_{x,t}(y)\right]
\end{equation*}
for $0<t\le T$.
Then, for a.e.~$x\in \Omega$,

\noindent $(i)$ the integral equation
\begin{equation}\label{int eq}
\underline g_{x,t}  =   (f_0)_x + \int_0^t (F[\underline g])_{x,s}\, ds
\end{equation}
has a unique solution $t\mapsto \underline{g}_{x,t}$ which belongs to $C([0,T], X_{[0,1]})$;\footnote{If $t\to\mu(t) $ is a continuous map from $[0,T]$ to $X_{[0,1]}$, for any Borel set $\mathcal B\subset[0,1]$, we set
$$ \Big(\int_0^t \mu(s)\, ds \Big)(\mathcal B) := \int_0^t  \mu(s) (\mathcal B) \, ds. $$}
\smallskip

\noindent $(ii)$ the measure $\underline{g}_{x,t}$ is a weak solution of the system
$$
\begin{cases}
\medskip
\partial_t \underline {g}_{x,t}(y)=\eta\chi \! \left[\partial_y 
\underline{A}_x(y,t)\displaystyle{\int}  \! P(t , \underline A_x(\xi,t), \underline{A}_x(y,t)) d\underline{g}_{x,t}(\xi)- \underline{g}_{x,t}(y)\right]\\
\underline{g}_{x,0}=(f_0)_x
\end{cases}
$$
in the sense of \eqref{weak sol measures}.
\end{lemma}

\begin{proof} 
First of all, we observe that for a.e.~$x\in\Omega$ and $s\in [0,T]$ and for all $ g\in X_{[0,1]}$,
\begin{equation}\label{dF=0}
\int d(F[g])_{x,s}=0.
\end{equation}
The assertion is obvious if $\chi(x,s)=0$. If $\chi(x,s)=1$, by Tonelli's theorem,
\begin{equation*}\begin{split}
\frac1\eta  \int d(F[g])_{x,s}& =\!
\int \!  \left( \displaystyle{\int}  \!\! P(s,\underline A_x(\xi,s), \underline{A}_x(y,s))  \partial_y \underline{A}_x(y,s))
dy         \right)  dg_{x,s}(\xi)
%\\&\qquad
 - \int dg_{x,s}(y)
 \\&= \int   \left( \displaystyle{\int}  \!\! P(s,\underline A_x(\xi,s), b)  
db         \right) \, dg_{x,s}(\xi)  - \int dg_{x,s}(y)=0.
\end{split}\end{equation*}

We set, for  a.e.~$x\in \Omega$ (from now on we fix such $x$),
$$ q_t:=e^{\int_0^t\eta(s) \chi(x,s)ds}  g_{x,t} \quad\text{for } t\in[0,T]. $$
Let $Y$ be the set of such $q$, i.e. $q\in Y$ if  
the map $t\mapsto e^{-\int_0^t\eta(s) \chi(x,s)ds} q_t$ belongs to $C([0,T], X_{[0,1]})$.
Then $Y$ naturally inherits a metric from $C([0,T], X_{[0,1]})$,
$$
d_Y(q_1,q_2):= \sup_{t\in(0,T)} \mathcal W_1\left(e^{-\int_0^t\eta(s) \chi(x,s)ds} (q_1)_t,e^{-\int_0^t\eta(s) \chi(x,s)ds} 
(q_2)_t\right),
$$
so $Y$ is a complete metric space.  

The equation for $ g$ translates into
$$
\partial_t q_t(y)=Lq_t(y):=\!\eta\chi(x,t) \partial_y \underline{A}_x(y,t)\int P(t, \underline A_x(\xi,t), \underline{A}_x(y,t)) 
dq_t(\xi)\ge 0,
$$
and the corresponding integral equation is
\begin{equation}\label{q integral}
q_t  =   (f_0)_x + \int_0^t Lq_s\, ds\quad  \text{for }%y\in [0,1], \ 
t\in [0,T].
\end{equation}

We consider the map
\begin{equation}\label{oct 8 eq:2}
 q \mapsto (f_0)_x + \int_0^t Lq_s\, ds.
\end{equation}
One easily checks that, by \eqref{dF=0}, for all $q\in Y$
\begin{equation}\label{dL=0}
\int dLq_t=0\quad \text{for }t\in [0,T], \qquad\quad (f_0)_x + \int_0^t Lq_s\, ds\in Y.
\end{equation}
If we show that for all $q_1,q_2\in Y$
\begin{equation}\label{d_Y}
J_0(q_1,q_2)\!:=d_Y\!\left(\! (f_0)_x + \int_0^t L(q_{1})_s,(f_0)_x + \int_0^t L(q_{2})_s\right)\!\le Cd_Y(q_1,q_2),
\end{equation}
it follows from a standard contraction argument that the map \eqref{oct 8 eq:2} has a unique fixed point in a sufficiently small interval 
$[0,\tau]$ and  that equation  \eqref{q integral} has a unique local solution $\underline q$ which can be continued in $[0,T]$.

To prove \eqref{d_Y} we use the characterisation of the $\cW_1$-distance given in 
Proposition \ref{W1}: 
\begin{equation}\label{d_Y bis}
d_Y(q_1,q_2)= \sup_{t\in(0,T)}\!\left[e^{-\int_0^t\eta(s) \chi(x,s)ds}\sup\left\{ \int \phi\, d(q_1-q_2)_t\; ; \; \phi \in \mathrm{Lip}_1([0,1],\mathbb R)\right\}\right].
\end{equation}
Hence  
$$
J_0(q_1,q_2)=\sup_{t\in(0,T)}\left[e^{-\int_0^t\eta(s) \chi(x,s)ds}\sup\left\{ I_\phi(t)\; ; \; \phi \in \mathrm{Lip}_1([0,1],\mathbb R)\right\}\right],
$$
where 
$$
I_\phi(t):=\int \phi\, d\int_0^t (L(q_1)_s-L(q_2)_s)\, ds
$$
and $L(q_1)_s-L(q_2)_s$ is given by
$$ 
\eta(s)\chi(x,s) \! \left( \displaystyle{\int}  \!\! P(s, \underline A_x(\xi,s), \underline{A}_x(y,s))\partial_y 
\underline{A}_x(y,s)d(q_1-q_2)_s(\xi)\!\right)\!dy.
$$
By Tonelli's Theorem,  $I_\phi(t)$ is equal to
$$
\begin{aligned}
&\int \!\!\phi(y) \!\!\int_0^t \!\!\left(
\!\!\eta\chi \! \left( \int  \!\! P(s, \underline A_x(\xi,s), \underline{A}_x(y,s))\partial_y 
\underline{A}_x(y,s) d(q_1-q_2)_s(\xi)\!\right)\!dy\!\!
\right)\! ds\\
&=\int_0^t\!\!\eta\chi  \!\! \left(\! \int_{\underline A_x(0,s)}^1 \!\!\phi(B_x(b,s)) 
 \!\! \left(\int  \!\! P(s, \underline A_x(\xi,s), b)
d(q_1-q_2)_s(\xi)\!\right)\!db
\right)\! ds\\
&=\int_0^t\!\!\eta\chi  \!\!\left( \!\int\!  \left(\int_{\underline A_x(0,s)}^1\! \! \!  \phi(B_x(b,s)) 
P(s, \underline A_x(\xi,s), b)db\! \right)\! 
d(q_1-q_2)_s(\xi)
\!\right)\! ds.\\
\end{aligned}
$$
By \eqref{dL=0}, $I_\phi(t)=0$ if $\phi$ is constant, so we may assume that $\phi(0)=0$.
Hence $|\phi|\le 1$ and, by \eqref{P lip},
$$
\begin{aligned}
&\left|\int_{\underline A_x(0,s)}^1 \phi(B_x(b,s)) \left( P(s, \underline A_x(\xi',s), b)-P(s, \underline A_x(\xi'',s), b)\right)db\right|\\
&\qquad\le \int_{\underline A(x,0,s)}^1 \left|P(s, \underline A_x(\xi',s), b)-P(s, \underline A_x(\xi'',s),b)\right|db
\le L |\xi'-\xi''|.
\end{aligned}
$$
Now \eqref{d_Y} follows from \eqref{d_Y bis}:  
$$
J_0(q_1,q_2)\le TL\max_{t\in [0,T]} \eta(t)\,d_Y(q_1,q_2).
$$

Setting 
$$
\underline g_{x,t}=e^{-\int_0^t\eta(s) \chi(x,s)ds}  \underline  q_t \quad\text{for } %y\in [0,1],\ 
t\in[0,T],
$$
we have  completed the proof of part $(i)$ of the lemma.

Fix an $x\in\Omega$ for which \eqref{int eq} and \eqref{q integral} (for $\underline q$) 
are valid. 
Since $P$ and $\underline A_x$ are continuous functions and the map $t\mapsto\underline  g_{x,t}$ is 
continuous in the weak$^*$ topology (and so is $t\mapsto\underline  q_t$), the map
$$
(y,t) \mapsto \int P(t, \underline A_x(\xi,t), \underline{A}_x(y,t)) d \underline  q_t(\xi)
$$
is continuous in $[0,1]\times [0,T]$. Hence $L(\underline q,\cdot)\in L^\infty((0,1)\times (0,T))$.

We set $\tilde q=\underline q-(f_0)_x$. By \eqref{q integral}
$$
\tilde q_t  = \int_0^t L(\tilde q_s+(f_0)_x)\, ds\quad  \text{for } %y\in [0,1], \ 
t\in [0,T].
$$
Since, by the boundedness of $L(\tilde q_s+(f_0)_x)(y)$, $t\mapsto \tilde q_t(y)$ is absolutely continuous in $[0,T]$ for a.e.~$y\in (0,1)$, 
this means  that
\begin{equation}\label{weak solution bis}
\int_0^1\!\psi(y,\tau)\tilde q_\tau(y)dy
=\!\iint_{(0,1)\times (0,\tau)}\!\!\left[ \psi_t(y,t) \tilde q_t  (y) \!+\! \psi(y,t)L((\tilde q_t\!-\!(f_0)_x)(y))\right]dydt
\end{equation}
for all $\tau\in (0,T]$ and $\psi \in L^\infty([0,1]\times [0,T])$ with $\psi_t\in L^\infty([0,1]\times [0,T])$. 

Finally let $\phi(y,t) $ be as in the first part of the proof (we recall that $x$ is fixed).
We substitute the function 
$\psi(y,t)=e^{-\int_0^t\eta(s) \chi(x,s)ds}\phi(y,t)$ into \eqref{weak solution bis}.
Since  
$$
\partial_t \psi(y,t)=e^{-\int_0^t\eta(s)\chi(x,s)ds} (-\eta \chi\phi(y,t)+\partial_t\phi(x,y,t),
$$
$\psi$ and $\partial_t\psi$ are  continuous with respect to $y$ and, by a straightforward calculation, \eqref{weak solution bis} 
transforms into
\begin{equation}\label{weak solution tris}
\begin{aligned}
&\int\phi(y,\tau)d\underline g_{x,\tau}(y)- \int\! \psi(y,0)d(f_0)_x(y)
=-\!\int_0^\tau \!\!\left[ \int\partial_t\psi(y,t)d(f_0)_x(y)\right]dt
\\&\qquad\qquad  
+\int_0^\tau \left( \int \phi(y,t)d(F[\underline g])_{x,t}(y)+\int\partial_t\phi(y,t))d\underline g_{x,t}(y)\right)dt\\
\end{aligned}
\end{equation}
for all $\tau\in (0,T]$. Since $\psi(y,0)=\phi(y,0)$, this implies that 
$\underline g_{x,t}(y)$ satisfies the equation of the system in the sense of \eqref{weak sol measures}: 
\begin{equation*}
\int\!\!\phi(y,\tau)d\underline g_{x,\tau}(y)- \int\!\! \phi(y,0)d(f_0)_x(y)
%\\&\qquad \qquad
=\int_0^\tau \!\left( \int\! \phi(y,t)d(F[\underline g])_{x,t}(y)+\int \!\partial_t \phi(y,t))d\underline g_{x,t}(y)\right)dt.
\end{equation*}
\end{proof}

Let $(\hat A,g,u) := (\hat A,g,u_1.\dots, u_N) \in \cX_{\tau,\rho}$.
By Lemma \ref{char febr 7}, $(\hat A,g,u)$ uniquely defines a function $\underline A% = \underline A(x,y,t)
\in L^{\infty}(\Omega; C([0,1]\times [0,\tau]; [0,1])$, and, by Lemma \ref{g},
 $ \underline A$ uniquely defines a measure  
$\underline g \in \mathcal L(\Omega;  C([0,T]; X_{[0,1]}))$.
Let $\underline u:= (\underline u_1,\dots, \underline u_N)$ be the weak solution of the
problem
\begin{equation}\label{system fixed point}
\begin{cases}
%\smallskip 
%\eps\partial_t \underline u_1 -d_1\!\Delta \underline u_1=F_1(\underline A, \underline g, u)\\
%\smallskip 
\eps\partial_t \underline  u_m - d_m\Delta  \underline  u_m =F_m(\underline A, \underline  g, u)&(1\leq m<N) \\
\smallskip
\eps\partial_t \underline  u_N=F_N(\underline A, \underline g, u),\\
\end{cases}
\qquad\text{in }Q_\tau=\Omega\times (0,\tau]
\end{equation}
%where $x\in Q_\tau=\Omega\times (0,\tau]$, 
with initial-boundary conditions
\begin{equation}\label{IC+BC fixed point}
	\begin{cases}
		\smallskip \underline u_i(x,0)=u_{0i}(x) &\text{if }x\in \Omega\\
		\smallskip \partial_n \underline u_i(x,t)=0&\text{if }x\in \partial \Omega_0, \ t>0\\
		\partial_n \underline u_i(x,t)=-\gamma_i \underline u_i(x,t)&\text{if }x\in \partial \Omega_1\times (0,\tau]
	\end{cases}
	\qquad (1\le i\le N).
\end{equation}
Here we have set 
\begin{equation*}
\begin{cases}
F_1(\underline A, \underline g, u)\!:=\!-\! \sigma_1u_1 \!- \!u_1\!\sum\limits_{j=1}^{N}a_{1,j}u_j
\!+\!C_{\mathcal F}\!\displaystyle {\int_0^1}(\mu_0\!+\! \underline A_x(\xi,t))(1\!- \!\underline A_x(\xi,t)) d\underline g_{x,t}(\xi) \\
F_m(\underline A, \underline  g, u)\!:=\!-\sigma_m u_m+\frac{1}{2}\sum\limits_{j=1}^{m-1}a_{j,m-j}u_ju_{m-j} 
-u_m\sum\limits_{j=1}^{N}a_{m,j}u_j   \\
F_N(\underline A, \underline g, u)\!:=\!\tfrac{1}{2}\sum\limits_{\substack{j+k\geq N \\ k,\,j<N}}a_{j,k}u_ju_{k}.\\
\end{cases}
\end{equation*}
Observe that $F_i\in L^\infty(\Omega\times [0,\tau])$ $(i=1,\dots,N)$ and its norm
only depends on the compact set $K\subset \R^n$ containing $(u_1,\cdots,u_N)$.
We also observe that system \eqref{system fixed point}-\eqref{IC+BC fixed point} consists of $N-1$ (uncoupled) scalar linear heat 
equations with linear boundary conditions and an ordinary differential equation.
Therefore  it has a unique weak solution $\underline u$.
More precisely, following \cite{nittka} and recalling that $\mathcal X_{\tau,\rho}$ denotes the closed ball of radius $\rho>0$ centered at $(y,f_0,u_0)$ in $\mathcal X_\tau$, we have that

\begin{proposition}[\cite{nittka}, Theorems 2.11, 3.2 and 3.3]\label{nittka} Let $(\hat A,\underline  g,u)\in X_{\tau,\rho}$. 
For all $1\le i< N$ there exists a unique $\underline u_i\in C([0,\tau]; L^2(\Omega))\cap L^2([0,\tau]; H^1(\Omega))$
such that
\begin{equation*}
\begin{split}
&\int_0^\tau 
\left[ \int_\Omega \nabla \underline u_i(x,s) \cdot \nabla \psi(x,s) \, dx 
+  \gamma_i\int_{\partial\Omega_1}
\underline u_i(x,s)\psi (x,s)\, d\sigma(x) \right]\, ds
\\&\qquad
=\iint_{Q_\tau} \underline u_i\partial_t \psi
+ \int_\Omega u_{0i}\psi(x,0)\, dx + \iint_{Q_\tau} F_i(\underline A,\underline   g, u)\psi
\end{split}
\end{equation*}
for all $\psi\in H^1([0,\tau]; H^1(\Omega))$, $\psi(x,\tau)=0$.
Let $\underline u_N(x,t) = u_{0N}(x)+\displaystyle{ \int_0^\tau} F_N(\underline A, \underline  g, u)\, ds$
and $\underline u=(\underline u_1,\cdots, \underline u_N)$.
Then $\underline u\in C(\overline Q_\tau; \mathbb R^N)$, $\underline u(\cdot,0)=\underline u_0$ and, for 
$1\le i\le N$, 
\begin{equation*}\begin{split}
\|\underline u_i\|_{C(\overline Q_\tau; \mathbb R)}
\le C \big\{
 \|u_{0i}\|_{L^\infty( \Omega)} + \|F_i\|_{L^r(Q_\tau; \mathbb R)}
\big\}\qquad \text{if }r> n, \quad \tfrac1r+\tfrac{n}{2r}<1.
\end{split}\end{equation*}
In particular
$
\|\underline u_i\|_{C(\overline  Q_\tau; \mathbb R)}
\le C \big\{
 \|u_{0i}\|_{C(\overline  \Omega)} + \|F_i\|_{C(\overline Q_\tau; \mathbb R)}
\big\}$.
\end{proposition}

Now we are ready to define the map to which we shall apply a contraction argument.
Let $\rho >0$ be fixed.
Using the notation $\underline A$ (Lemma \ref{char febr 7}), $\underline g$ (Lemma \ref{g}) and $\underline u$ (Proposition \ref{nittka})
introduced above, we set
\begin{equation}\label{the map H}
\cH(\hat A,g,u) := (\underline A,\underline g,\underline u)\quad\text{for }(\hat A,g,u)\in \cX_{\tau,\rho}.
\end{equation}
Let $\cT_d$ denote the metric topology of $X_{\tau,\rho}$ and $\cT$ the weaker topology on $\cX_{\tau,\rho}$ 
which is obtained by endowing 
$L^{\infty}(\Omega; C([0,1]\times [0,\tau]; [0,1]))$ with the $L^1$-topology on $\Omega\times [0,1]\times [0,\tau]$.

\begin{proposition}\label{contr 1} Let $\rho >0$ be fixed and let $\cH(\hat A,g,u)$ be defined by \eqref{the map H}. 
If $\tau>0$ is sufficiently small,  then
$\cH: \cX_{\tau,\rho} \to \cX_{\tau,\rho}$, 
%if $(\hat A_n, g_n, u_n) \to (\hat A, g, u)$ as $n\to \infty$ in the topology $\cT_d$, then
$(\underline A_n, \underline g_n, \underline u_n) \to (\underline A, \underline g, \underline u)$
%as $n\to \infty$ 
in $\cT$ if $(\hat A_n, g_n, u_n) \to (\hat A, g, u)$ in $\cT_d$,
and $\cH$ is a contraction on $\cH(X_{\tau,\rho})$.
\end{proposition}

\begin{proof} First we prove that $\cH(\cX_{\tau,\rho})\subset \cX_{\tau,\rho}$ if $\tau$ is sufficiently small.

By Proposition \ref{nittka}, $\|u(\cdot,t) - u_0\|_{C(\overline \Omega; \mathbb R^N})\to 0$ as $t\to 0^+$, so 
it remains to show that,  as $t\to 0^+$, 
\begin{equation}\label{A,g; t to 0}
\sup_{x\in \Omega,\ 0\le y\le 1}|\underline A_x(y,t)-y|\to 0, \qquad
\sup_{x\in \Omega}\| \underline g_{x,t}-(f_0)_x\|_{X_{[0,1]}}\to 0.
\end{equation}
Since, by \eqref{cauchy A} and assumptions
\eqref{Ga global} and \eqref{Sa global}, %\eqref{G uniform}, \eqref{S uniform} we have that
\begin{equation*}\begin{split}
|\underline A_x(y,t)-y|& \le
\int_0^\tau 	\left\{\int
%\Big\{\displaystyle {\int}
\cG_x(\underline A_x(y,s), \hat A_x(\xi,s))
		\,dg_{x,s}(\xi)+\cS(x, \underline A_x(y,s),u(x,s)) 
\right\} %\Big\}
\, ds
\\&
%\hphantom{\le\int_0^\tau 	\Big\{\displaystyle {\int}
%\\&
\le
C_1 \int_0^\tau |\underline A_x(y,s)-y|\; ds + C_2\tau,
\end{split}\end{equation*}
\eqref{A,g; t to 0}$_1$ follows from Gronwall's Lemma.
On the other hand, \eqref{A,g; t to 0}$_2$ easily follows from Lemma \ref{g}$(i)$ and its proof.

To prove the $(\cT_d,\cT)$-continuity of $\cH$, let $\hat A_n, \hat A\in 
L^{\infty}(\Omega; C([0,1]\times [0,\tau];[0,1])$
be such that $(\hat A_n, g_n, u_n) \to (\hat A, g, u)$ in $ \cX_{\tau,\rho}$
as $n\to \infty$. % with respect to$\cT_d$.
We must show that $A_n\to A$ in $L^1(\Omega\times [0,1]\times [0,\tau])$.

By the Dominated Convergence Theorem, this follows if 
\begin{equation}\label{conv a.e.}
\underline A_n  \to \underline A \quad \text{a.e.~in }\Omega\times [0,1]\times [0,\tau]\quad\mbox{as $n\to\infty$}.
\end{equation}

To prove \eqref{conv a.e.} we observe that 
\begin{equation}\label{0}\begin{split}
|  (\underline A_n)_x(y,t) - & \underline A_x(y,t)|
\\ \le &
\int_0^t\!\left|  \int\!\left[\cG_x((\underline A_n)_x(y,s), (\hat A_n)_x(\xi,s))-\cG_x(\underline A_x(y,s), \hat A_x(\xi,s))\right]
d (g_n)_{x,s}(\xi)\right|\!ds
\\&
+ \int_0^t\left| \int \cG_x(\underline A_x(y,s), \hat A_x(\xi,s))d (g_n-g)_{x,s}(\xi)\right|ds
\\&
+ \int_0^t\left|\cS(x, (\underline A_n)_x(y,s), u_n(x,s)) - \cS(x, (\underline A_n)_x(y,s), u(x,s))\right|\, ds
\\&
+ \int_0^t\left|\cS(x, (\underline A_n)_x(y,s), u(x,s)) - \cS(x, \underline A_x(y,s), \underline u(x,s))\right|\, ds
\\=:&\ 
I_1+I_2+I_3+I_4,
\end{split}\end{equation}
where $I_j= I_j(x,y,t)$ for $j=1,2,3,4$.
It follows easily from \eqref{Sa global} that 
$$
I_3\le C_\rho t\, \sup_{x\in \Omega, 0\le s\le \tau} |u_n(x,s) - u(x,s)|
\le
C_\rho t\, d((\hat A_n,g_n,u_n), (A,g,u)),
$$
\begin{equation*}\begin{split}
I_4  & \le C_\rho t\! \sup_{x\in \Omega, \ 0\le y\le 1, \ 0\le s\le \tau}\! |(\hat A_n)_x(y,s)) \!- \!\hat A_x(y,s)|
%\\& 
\!\le\!
C_\rho t\, d((\hat A_n,g_n,u_n), (A,g,u)).
\end{split}\end{equation*}
By \eqref{Ga global},
\begin{equation*}\begin{split}
I_1  \le & \, C_\rho  \int_0^t \Big\{ \int \big| (\underline A_n)_x(y,s) - \underline A_x(y,s)\big| d(g_n)_{x,s}(\xi)\Big\} \, ds
\\&
%\hphantom{xxx} 
+ C_\rho  \int_0^t \Big\{ \int \big|(\hat A_n)_x(\xi,s))-\hat A_x(\xi,s)) \big| d (g_n)_{x,s}(\xi)\Big\} \, ds
\\
= &\, C_\rho  \int_0^t \big| (\underline A_n)_x(y,s) - \underline A_x(y,s)\big|  \, ds
\\&
+ C_\rho  \int_0^t \Big\{ \int \big|(\hat A_n)_x(\xi,s))-\hat A_x(\xi,s)) \big| d (g_n)_{x,t}(\xi)\Big\} \, ds
\\
\le &\, C_\rho  \int_0^t \big| (\underline A_n)_x(y,s) - \underline A_x(y,s)\big|  \, ds
+ C_\rho t\, d((\hat A_n,g_n,u_n), (A,g,u)).
\end{split}\end{equation*}
Thus, by Gronwall's inequality, Proposition \ref{wasserstein} and the Dominated Convergence Theorem,
\begin{equation}\label{continuity 2}\begin{split}
\Big| & ( \underline A_n)_x(y,t) - \underline A_x(y,t)\Big| 
\le I_2(x,y,\tau) + C_\rho  \tau \, d((\hat A_n,g_n,u_n), (A,g,u)) \to 0
\end{split}\end{equation}
as $n\to\infty$. This proves \eqref{conv a.e.}.

It remains to prove that $\cH$ is a contraction on $\cH(X_{\rho,\tau})$ if $\tau$ is 
small enough. Let $(\hat A^1,g^1,u^1)$, $(\hat A^2,g^2,u^2) \in X_{\rho,\tau}$. Repeating verbatim the arguments leading to 
\eqref{continuity 2}, we obtain that
\begin{equation}\label{continuity 3}\begin{split}
\left|  [\underline A^1_x- \underline A^2_x](y,t)\right| \le  
&\int_0^t\left| \int\cG_x(\underline A^2_x(y,s), \hat A^2_x(\xi,s))d 
(g^1-g^2)_{x,s}(\xi)\right| ds\!  %\big)
\\&
+\! C_\rho  \tau \, d((\hat A^1,g^1,u^1), (\hat A^2,g^2,u^2)).
\end{split}\end{equation}
Since $(\hat A^2,g^2,u^2)\in\cH(X_{\rho,\tau})$, 
it follows from Remark \ref{stime A} 
that $\hat A^2_x(\xi,s)$ and, by \eqref{Ga global}, \\ 
$\cG_x(\underline A^2_x(y,s), \hat A^2_x(\xi,s))$
are Lipschitz continuous in $\xi$, uniformly with respect to $x$ and $s$.
Thus, by Proposition \ref{W1},
\begin{equation}\label{continuity 4}\begin{split}
\left|  \underline A^1_x(y,t) - \underline A^2_x(y,t)\right| 
&\le C_\rho \int_0^t \cW(g^1_{x,s}, g^2_{x,s}) \, ds  
\le  
C_\rho  \tau \, d((\hat A^1,g^1,u^1), (A^2,g^2,u^2)).
\end{split}\end{equation}

Consider now
$\cW_1(\underline g^1_{x,t}, \underline g^2_{x,t})$.
In view of the definition of $\underline g^1,\underline g^2$,
we may repeat verbatim the arguments in the proof of Lemma \ref{g} and obtain that
\begin{equation}\label{May 27 eq:1}\begin{split}
\cW_1  (\underline g^1_{x,t}, \underline g^2_{x,t})
& \le C\max_{[0,T]} \eta \,t\; \sup_{0\le s\le t}\cW_1\big(
g^1_{x,s} ,g^2_{x,s}\big)
%\\&\le Ct\,d((\hat A^1,g^1,u^1),(\hat A^2,g^2,u^2))
\\&
\le 
C\tau \,d((\hat A^1,g^1,u^1),(\hat A^2,g^2,u^2)).
\end{split}\end{equation}

Finally, we estimate $\sup\limits_{\Omega\times [0,\tau]} |\underline u^1- \underline u^2|$.
Set
$\underline U= \underline u^1-\underline u^2$ and 
$\underline U = (\underline U_1, \dots, \underline U_N)$.
Then $\underline U$ is a weak solution (in the sense of Proposition \ref{nittka}) of a system 
similar to \eqref{system fixed point}-\eqref{IC+BC fixed point}, with $F_j$ replaced by
$
\tilde F_j:= F_j(\underline A^1,\underline g^1,u^1) - F_j(\underline A^2, \underline g^2,u^2),
$
$j=1,\dots, N$, and $u_0$ by $\underline U(x,0)=0$. 
By Proposition \ref{nittka}, 
\begin{equation}\label{april 28 eq:1}\begin{split}
\|\underline U \|_{C(\overline  \Omega\times [0,\tau]; \mathbb R)}
\le C 
\sum_{i=1}^N \|\tilde F_i\|_{C(\overline \Omega\times [0,\tau]; \mathbb R)}.
\end{split}\end{equation}
If $k>1$, $F_k$ is a polynomial in the components of $u$ and, since
% so that, keeping into account that
$u_1,u_2$ are uniformly bounded by $\rho$ in $\Omega\times [0,\tau]$,
\begin{equation}\label{april 28 eq:2}
\|\tilde F_k\|_{C(\overline \Omega\times [0,\tau]; \mathbb R)}  \le C_\rho
\sum_i  \| u_i^1 - u_i^2\|_{C(\overline \Omega\times [0,\tau]; \mathbb R)}\quad\text{if }k>1.
\end{equation}
The same argument applies to the polynomial terms of $\tilde F_1$, so we are left with the estimate of
\begin{equation*}
I\!=\left\| \int_0^1\!\!(\mu_0\!+\!  \underline A^1_x(\xi,t))  (1\!-\! \underline A^1_x(\xi,t))\, d\underline g^1_{x,t}(\xi)
\!-\!\!\int_0^1\!\!(\mu_0\!+\! \underline A^2_x(\xi,t))(1\!-\! \underline A^2_x(\xi,t))\, d\underline g^2_{x,t}(\xi) \right \|_{C(\overline \Omega\times [0,\tau]; \mathbb R)}\!\! .
\end{equation*}
Arguing as above, 
\begin{equation*}\begin{split}
I \le J_1+J_2
:= &\int_0^1\big| (\mu_0+  \underline A^1_x(\xi,t))  (1- \underline A^1_x(\xi,t))
 -  (\mu_0+  \underline A^2_x(\xi,t))  (1- \underline A^2_x(\xi,t))\big|
\, d\underline g^1_{x,t}(\xi)
\\&+
\Big|  \int_0^1 (\mu_0+ \underline A^2_x(\xi,t))  (1- \underline A^2_x(\xi,t)) \,d(\underline g^1-\underline g^2)(\xi)  \Big|.
\end{split}\end{equation*}
Repeating the arguments that yield the estimate \eqref{continuity 3}, 
we have that
\begin{equation*}\begin{split}
J_1  & \le C_\rho\, \sup_{x,\xi,t}
\big |\underline A^1_x(\xi,t)) - \underline A^2_x(\xi,t)) \big|
 \int_0^1 d \underline g^1_{x,t}(\xi)
 \\&
  = C_\rho\, \sup_{x,\xi,t} \big |\underline A^1_x(\xi,t)) - \underline A^2_x(\xi,t))\big| 
%\\&
  \le C_\rho  \tau \, d((\hat A^1,g^1,u^1), (A^2,g^2,u^2)).
\end{split}\end{equation*}
Concerning $J_2$, by Remark \ref{stime A} 
the map $\xi\to  (\mu_0+ \underline A^2_x(\xi,t))  (1- \underline A^2_x(\xi,t))$
is uniformly Lipschitz continuous. Thus, by Proposition \ref{W1} and \eqref{May 27 eq:1},
\begin{equation*}
J_2 \le C_\rho \cW_1(\underline g^1,\underline g^2)
\le C_\rho\tau \,d((\hat A^1,g^1,u^1),(\hat A^2,g^2,u^2)).
\end{equation*}
Combining the estimates of $I,J_1,J_2$
with \eqref{april 28 eq:2} and \eqref{april 28 eq:1}, we obtain that
\begin{equation}\label{april 28 eq:3}
\|\underline u_1 - \underline u_2 \|_{C(\overline \Omega\times [0,\tau]; \mathbb R)}
\le C_\rho\tau \,d((\hat A^1,g^1,u^1),(\hat A^2,g^2,u^2)).
\end{equation}

It follows from \eqref{continuity 4}, \eqref{May 27 eq:1} and \eqref{april 28 eq:3}
that $\cH$ is a contraction on $\cH(\cX_{\rho, \tau})$
if $\tau$ is small enough.
\end{proof}

To complete the proof of Theorem \ref{theorem local existence}, we need a minor modification of the classical Banach-Caccioppoli fixed point theorem:
\begin{proposition}[Fixed Point Theorem]
Let $(X,d)$ be a complete metric space and let $\cT_d$ be the topology induced by $d$. Let $\cT$ be a Hausdorff topology on $X$ which is weaker than $\cT_d$. If $\cH:X\to X$ is a contraction on $\cH(X)$ which is $(\cT_d, \cT)$-continuous, then $\cH$ has a unique fixed point.
\end{proposition}
\begin{proof}
We start carrying out the standard iteration procedure 
\begin{equation}\label{iteration}
x_{n+1}= \cH(x_n),
\end{equation}
starting from a point $x_0\in \cH(X)$, so that $x_n\in \cH(X)$ for all $n\ge 0$. As usual, by the completeness of $(X,d)$, we may assume that $x_n\to \bar x\in X$ as $n\to \infty$. When $\cH$ is a contraction on all of $X$ (and hence  in particular is Lipschitz continuous from $X$ to $X$) we can conclude the proof taking the limit as $n\to\infty$ in \eqref{iteration}. In our case the argument has to be slightly adapted: on one side, $x_{n+1}\to\bar x$ as $n\to\infty$ with respect to the topology $\cT$ (since it is weaker that $\cT_d$), on the other hand $\cH (x_{n})\to  \cH (\bar x)$ as $n\to\infty$ with respect to the topology $\cT$, since $\cH$  is $(\cT_d, \cT)$-continuous, Thus, by \eqref{iteration} we can conclude that $\bar x = \cH (\bar x)$ by the uniqueness of the limit in $\cT$.
\end{proof}

\begin{proof}[Proof of Theorem \ref{theorem local existence}.]
By Proposition \ref{contr 1} and the Fixed Point Theorem, 
system \eqref{system rewritten}-\eqref{IC+BC} has a unique solution (in the sense of Definition \ref{definition of solution})
in $[0,\tau]$ for sufficiently small values of $\tau$ if we show the nonnegativity of $u_i$: 
\begin{equation}\label{u>0}
u_i \ge 0 \quad\text{in }\Omega\times [0,\tau] \qquad (i=1,\dots,N).
\end{equation}
If $i=N$, \eqref{u>0} is trivially satisfied. If $1\le i<N$, \eqref{u>0} formally follows from the maximum principle. 
Below we make this precise if $i=1$. If $i>1$ the proof is even easier.

Since $f= C_{\mathcal F}\displaystyle {\int_0^1}(\mu_0+\hat A_x(\xi,t))(1-\hat A_x(\xi,t))\, dg_{x,t}(\xi)$
is nonnegative and belongs to $L^{\infty}(Q_T)$, there exists a sequence of smooth nonnegative functions $(f_k)_{k\in \mathbb N}$ 
converging to $f$ in $L^r(Q_T)$, where $r> n$ and $\frac1r+\frac{n}{2r}<1$. 
We also approximate 
$h=\sum_{j=1}^{N}a_{1,j}u_j\in C(\overline Q_T)$ uniformly by 
smooth functions $h_k$.   Let $v_k$ be the unique smooth solution of 
 \begin{equation}\label{CB k}
	\begin{cases}
 \eps\partial_tv_k=d_1\!\Delta v_k-v_k h_k + f_k&\text{in }Q_\tau\\
	\smallskip v_k(x,0)=u_{01}(x) & \text{if }x\in \Omega\\
		\smallskip \partial_n v_k (x,t)=0 & \text{if }x\in \partial \Omega_0, \ t>0\\
		\partial_n v_k (x,t)=-\gamma_i v_k (x,t)&\text{if }x\in \partial \Omega_1,\ t>0. 
	\end{cases}
\end{equation}
Since $\gamma_1>0$, $f_k\ge 0$ in $Q_\tau$ and $u_{01}\ge 0$ in $\Omega$, it follows from the maximum principle
that $v_k\ge $ in $Q_\tau$.  

On the other hand $w_k:=u_1-v_k$  is a weak solution of
 \begin{equation}\label{CB k bis}
	\begin{cases}
 \eps\partial_tw_k=d_1\!\Delta w_k-w_k h_k +f- f_k &\text{in }Q_\tau\\
	\smallskip w_k(x,0)=0 & \text{if }x\in \Omega\\
		\smallskip \partial_n w_k (x,t)=0 & \text{if }x\in \partial \Omega_0, \ t>0\\
		\partial_n w_k (x,t)=-\gamma_i w_k (x,t)&\text{if }x\in \partial \Omega_1,\ t>0, 
	\end{cases}
\end{equation}
and it follows from \cite{nittka}, Theorem 3.2, that $v_k \to u_1$ uniformly on $\overline Q_\tau$. Therefore
also $u_1\ge 0$ in $Q_\tau$.
\end{proof}

\section{Global existence}\label{section global existence}
In this section we complete the proof of Theorem \ref{theorem main result} by showing that the local solution 
of problem \eqref{system rewritten}-\eqref{IC+BC}, constructed in the previous section, can be continued to
the whole interval $[0,T]$. We recall that
problems \eqref{system rewritten}-\eqref{IC+BC} and \eqref{complete system}-\eqref{IBC} are equivalent, 
as we have shown in section \ref{characteristics}.

Arguing by contradiction we suppose that the maximal interval of existence is $[0,\tau^*)$ for 
some $\tau^*<T$. 

\subsection*{A priori estimate for $u(x,t)$}
Since 
\begin{equation} \label{boundedness F script}
	C_\cF\int_0^1(\mu_0+A_x(\xi,t))(1-\hat A_x(\xi,t))\,dg_{x,t}(\xi)\le C_1 \quad \text{in } \Omega\times [0,\tau^*)
\end{equation}
for some constant $C_1$, it follows formally from the maximum principle that 
$$ u_1(x,t)\le\sup_\Omega u_{01}+C_1t \quad \text{for } x\in \Omega,\ 0\le t<\tau^*. $$
Similarly, if $u_1,\cdots, u_{m-1}$ are bounded in $L^\infty(\Omega\times [0,\tau^*))$ for some $1<m< N$, then
$$ \frac{1}{2}\sum\limits_{j=1}^{m-1}a_{j,m-j}u_ju_{m-j} \le C_m \quad \text{in } \Omega\times [0,\tau^*) $$
for some constant $C_m$, and it follows formally from the maximum principle that
$$ u_m(x,t)\le\sup_\Omega u_{0m}+C_mt \quad \text{for } x\in\Omega,\ 0\le t<\tau^*. $$
In both cases the use of the maximum principle is justified as in the proof of \eqref{u>0}.

The boundedness of $u_N$ in $\Omega\times [0,\tau^*)$ follows from that of $u_1,\cdots, u_{N-1}$, so we have shown that, for some $C_u>0$,
\begin{equation}\label{estimate of u}
	|u|\le C_u \quad \text{in } \Omega\times [0,\tau^*).
\end{equation}

\subsection*{Existence of $\lim_{t\to \tau^*}A_x(y,t)=:A_x(y,\tau^*)$.}
Arguing as in the proof of Lemma \ref{char febr 7} we obtain that $A_x(y,t)$ and  $v_x(A_x(y,t),t)$ are Lipschitz continuous with respect to $y$, uniformly with respect to $x\in \Omega$ and $t\in [0,\tau^*)$. By the boundedness of $v_x(A_x(y,t),t)$, the map $t\mapsto A_x(y,t)$ is Lipschitz continuous on $[0,\tau^*)$. Hence $A_x(y,\tau^*):=\lim_{t\to \tau^*}A_x(y,t)$ exists and is   Lipschitz continuous with respect to $y$, uniformly with respect to $x\in \Omega$.

\subsection*{Existence of $\lim_{t\to \tau^*} g_{x,t}=:  g_{x,\tau^*}$.}
We repeat verbatim the arguments of the proof of Lemma \ref{g} and we obtain that the map $t \mapsto  g_{x,t}$ is Lipschitz continuous from $[0,\tau^*)$ to $X_{[0,1]}$ endowed with Wasserstein metric $\cW_1$.

\subsection*{Existence of $\lim_{t\to \tau^*} u (x,t) =:  u(x,\tau^*)$.}
In view of \eqref{boundedness F script} and \eqref{estimate of u}, it follows from standard regularity theory for weak solutions of parabolic equations (see e.g. \cite[Theorem 1, page 111]{rothe}) that $u$ is uniformly (H\"older) continuous in $\Omega\times [0,\tau^*)$. Hence $u$ can be extended to  $\Omega\times [0,\tau^*]$ as a continuous function.

Hence we can apply the local existence theorem to the ``initial'' functions $g_{x,\tau^*}$ and $u(x,\tau^*)$, and obtain a solution in $[\tau^*,\tau_1]$ for some $\tau_1\in [\tau^*,T]$. Therefore $[0,\tau^*)$ is not the maximal interval of existence and we have found a contradiction.

\section*{Acknowledgments}
The authors would like to express their gratitude to MD Norina Marcello for many stimulating and fruitful discussions over several years.

B.~F. and M.~C.~T. are supported by the University of Bologna, funds for selected research topics, and by MAnET Marie Curie
Initial Training Network. B.~F. is supported by 
GNAMPA of INdAM (Istituto Nazionale di Alta Matematica ``F. Severi''), Italy, and by PRIN of the MIUR, Italy.

A.~T. is member of GNFM (Gruppo Nazionale per la Fisica Matematica) of INdAM (Istituto Nazionale di Alta Matematica ``F. Severi''), Italy.

\bibliographystyle{amsplain}
\bibliography{BmFbTmcTa-alzheimer_analysis-R1}

\appendix

\section{Probability measures and Wasserstein metrics}
\label{app:measures}
Throughout this appendix, $X$ denotes a {\it complete separable metric space}, with metric $d$. A positive Borel measure $\mu$ on $X$ such that $\mu(X)=1$ is said a probability measure, and we write $\mu\in \cP(X)$. Every $\mu\in \cP(X)$ is a Radon measure (see \cite{AGS}).

\begin{definition}[Push forward of measures]
\label{def:push_forward}
Let $\cB(X)$ be the Borel $\sigma$-algebra of subsets of $X$ and $\phi:X\to X$ a Borel measurable function, i.e. one such that $\phi^{-1}(E)\in\cB(X)$ for every $E\in\cB(X)$. Let moreover $\mu\in\cP(X)$. The \emph{push forward} of $\mu$ through $\phi$ is the measure $\nu\in\cP(X)$, denoted by $\nu=\phi_\#\mu$, such that
$$ \nu(E):=\mu(\phi^{-1}(E)), \quad \forall\,E\in\cB(X). $$
Equivalently, the measure $\nu$ can be characterised by
$$ \int_X f(x)\,d\nu(x)=\int_X f(\phi(x))\,d\mu(x) $$
for every bounded Borel function $f$ defined on $X$.
\end{definition}

\begin{definition}[Wasserstein distances]
Let $p\geq 1$ and $\mu,\,\nu\in\cP(X)$ be such that
$$ \int_X d(x,\bar{x})^p\,d\mu(x)<+\infty, \qquad \int_X d(x,\bar{x})^p\,d\nu(x)<+\infty $$
for some $\bar{x}\in X$. The \emph{$p$-th Wasserstein distance} between $\mu$ and $\nu$ is the number denoted by $\cW_p(\mu,\nu)$ and defined by
$$ \cW_p^p(\mu,\nu):=\inf\left\{\iint_{X^2}d(x,y)^p\,d\gamma(x,y)\,:\,\gamma\in\Gamma(\mu,\nu)\right\}, $$
where $\Gamma(\mu,\nu)\subset\cP(X^2)$ is the set of all \emph{transference plans} between $\mu$ and $\nu$, i.e. the set of measures $\gamma\in\cP(X^2)$ whose \emph{marginals} are $\mu$, $\nu$, respectively.
\end{definition}

\begin{proposition}[\cite{AGS}, Proposition 7.1.5] \label{moment}
If $\mu\in\cP(X)$ has compact support then for any $\bar x\in X$ and $p\geq 1$
$$ \int_X d(x,\bar x)^p\,d\mu(x)<+\infty. $$
In particular, $\mu$ has finite $p$-moment. We shall write $\mu\in \cP_p(X)$. Endowed with the Wasserstein $p$-distance $\cW_p$, $\cP_p(X)$ is a complete metric space. 
\end{proposition}

\begin{proposition}[Kantorovich-Rubinstein duality, cf. \cite{AGS}, Eq. (7.1.2)]
\label{W1}
If $\mu, \nu \in \cP_1(X)$ have compact support, then
$$ \cW_1(\mu,\nu)=\sup\left\{\int_X\phi\,d(\mu-\nu)\,:\,\phi\in\operatorname{Lip}_1(X,\mathbb R)\right\}, $$
where $\operatorname{Lip}_1(X,\mathbb R)$ is the space of Lipschitz continuous functions $\phi:X\to\R$ with Lipschitz constant not greater than $1$.
\end{proposition}

\begin{definition} Let $(\mu_n)_{n\in\mathbb N}$ be a sequence in $\cP(X)$.
We say that
\begin{enumerate}[(i)]
\item $\mu_n\to\mu$ \emph{narrowly} if for any bounded continuous function $f$
$$ \int_X f\,d\mu_n\to\int_X f \,d\mu \qquad \text{as } n\to\infty $$
\item $\mu_n\to\mu$ \emph{$\text{weakly}^*$} if for any compactly supported continuous function $f$
$$ \int_X f\,d\mu_n\to\int_X f\,d\mu \qquad \text{as } n\to\infty. $$
\end{enumerate}
\end{definition}

\begin{remark}\label{narrow - weak}
Obviously, narrow convergence implies $\text{weak}^*$ convergence, and narrow and weak$^*$ convergence are equivalent if $X$ is compact.
\end{remark}

\begin{proposition}\label{wasserstein}
Let $X$ be a separable metric space. Let $(\mu_n)_{n\in\mathbb N}$ be a sequence in $\cP_p(X)$. We have:
\begin{enumerate}[(i)]
\item if $\cW_p(\mu_n,\mu)\to 0$ as $n\to\infty$, then $\mu_n\to\mu$ as $n\to\infty$ $\text{weakly}^*$;
\item suppose there exist a compact set $K$ such that $\mathrm{supp}\;\mu_n\subset K$ for all $n\in\mathbb N$ and an open set $\mathcal O$ satisfying
$$ K\subset\cO \qquad \text{and} \qquad X\setminus \cO\neq \emptyset. $$
Then
$$ \cW_p(\mu_n,\mu)\to 0 \qquad \text{as } n\to\infty $$
if and only if $\mu_n\to\mu$ as $n\to\infty$ $\text{weakly}^*$ (or, equivalently, narrowly).
\end{enumerate}
\end{proposition}
\begin{proof}
We apply \cite[Proposition 7.1.5]{AGS}. We have but to prove that the $\mu_n$'s have uniformly integrable $p$-moments. By \cite[Lemma 5.1.7]{AGS} the assertion will follow by showing that
$$ \lim_{n\to\infty}\int_X f(x)\, d\mu_n(x) = \int_X f(x)\, d\mu(x) $$
for any continuous real function $f$ such that
\begin{equation}\label{p growth}
	|f(x)| \le A + Bd(x,\bar x)^p \qquad A,B>0,\ \bar{x}\in X \text{ fixed}.
\end{equation}
Take now a continuous map $f:X\to\R$ satisfying \eqref{p growth}. By Urysohn's lemma we can easily construct a continuous function $\tilde f$ such that
$$ \supp\tilde f\subset\cO \qquad \text{and} \qquad \tilde f\equiv f\ \text{in } K. $$
Thus
\begin{equation*}
\begin{split}
\lim_{n\to\infty} & \int_X f(x)\, d\mu_n(x) = \lim_{n\to\infty}\int_K f(x)\, d\mu_n(x) = \lim_{n\to\infty}\int_K \tilde f(x)\, d\mu_n(x)
\\&
= \lim_{n\to\infty}\int_X \tilde f(x)\, d\mu_n(x) = \int_X \tilde f(x)\, d\mu(x) = \int_K \tilde f(x)\, d\mu(x)
\\&
= \int_K f(x)\, d\mu(x) = \int_X  f(x)\, d\mu(x). \qedhere
\end{split}
\end{equation*}
\end{proof}

\begin{remark}
If $X$ is compact, then the assertion is trivial. Indeed, we have already pointed out that narrow convergence and $\text{weak}^*$ convergence are equivalent on compact metric spaces. Thus we can apply \cite[Proposition 7.1.5]{AGS}. Indeed the $\mu_n$'s have uniformly integrable $p$-moments, by \cite[Lemma 5.1.7]{AGS}.
\end{remark}

\begin{proposition}\label{push}
Let  $X,Y$ be  complete separable metric spaces. In addition, let $X$ be compact and assume
that for any compact set $K\subset Y$ there exists an open set $\cO$ such that
$$ K\subset \cO  \qquad \text{and} \qquad Y\setminus \cO\neq \emptyset. $$
Let $(\mu_n)_{n\in\mathbb N}$ be a sequence in $\cP_p(X)$. Let $\Phi_n: X\to  Y$ be a sequence of continuous injective (and hence open) maps that converges uniformly to a continuous injective (and hence open) map $\Phi: X\to  Y$. Then, if $p>0$
$$ \lim_{n\to \infty} \cW_p(\mu_n,\mu)=0 \qquad\mbox{if and only if}\qquad  \lim_{n\to \infty} \cW_p((\Phi_n)_\#\mu_n, \Phi_\#\mu)=0. $$
\end{proposition}
\begin{proof}  
By Proposition \ref{wasserstein} and  \cite[Remark 5.1.5]{AGS}, the sequence $(\mu_n)_{n\in\mathbb N}$ is tight. By \cite[Lemma 5.2.1]{AGS} 
$$ \lim_{n\to \infty}\cW_p(\mu_n,\mu)=0 \quad \Rightarrow \quad (\Phi_n)_\#\mu_n\to \Phi_\#\mu $$
narrowly as $n\to\infty$. Set
set 
$$ K_0:= \overline{\{y\; ; \; d(y, \Phi(K))<\eps\}}. $$
If $n>\bar n$, then $\mathrm{supp}\, (\Phi_n)_\#\mu_n \subset \Phi_n(K)\subset K_0$ that is compact. By assumption, there is an open set $\cO_0$ such that 
$$ K_0\subset \cO_0  \qquad \text{and} \qquad Y\setminus \cO_0\neq \emptyset. $$
Thus, by Proposition \ref{wasserstein}, $\lim_{n\to \infty} \cW_p((\Phi_n)_\#\mu_n, \Phi_\#\mu)=0$. This proves the first part of the statement.

Suppose now $ \lim_{n\to \infty} \cW_p((\Phi_n)_\#\mu_n, \Phi_\#\mu)=0$.

We notice now that the sequence $(\mu_n)_{n\in\mathbb N}$ in $\cP(X)$ is tight (again by Remark \ref{narrow - weak}), and hence, by \cite[Theorem 5.1.3]{AGS}, is relatively compact with respect to the narrow convergence. Therefore, there exists a subsequence $(\mu_{n_j})_{j\in\mathbb N}$ converging narrowly to $\nu\in \cP(X)$. By Proposition \ref{wasserstein} $\lim_{j\to \infty} \cW_p(\mu_{n_j},\nu) =0$, and then $ \lim_{n\to \infty} \cW_p((\Phi_{n_j})_\#\mu_n, \Phi_\#\nu)=0$ (by the first part of the present proposition). Thus the uniqueness of the Wasserstein limit yields $\Phi_\#\nu = \Phi_\#\mu$ and eventually $\nu =\mu$, i.e. $\lim_{j\to \infty} \cW_p(\mu_{n_j},\mu) =0$. A standard argument in metric spaces makes possible to recover the limit for the full sequence $(\mu_n)_{n\in\mathbb N}$.
\end{proof}

\begin{corollary} \label{cont meas}
Let $X,\,Y$ be complete separable metric spaces satisfying the assumption of Proposition \ref{push}. If $I\subset \mathbb R$ is an interval, let $\Phi: X\times I\to Y$ be a continuous map such that for any $t\in I$ the map $x\to \Phi(x,t)$ is injective and open.

If $t\in I$, let $\mu(t)\in \cP(X)$ such that $\supp \mu(t)\subset K$ for all $t\in I$, where $K\subset X$ is a compact as in Proposition \ref{wasserstein}.
 
Then $t\mapsto\mu(t)$ is continuous (with respect to the Wasserstein topology) if and only if $t\mapsto\Phi(\cdot,t)_\#\mu(t)$ is continuous (with respect to the Wasserstein topology).
\end{corollary}

\end{document}